\documentclass[11pt]{amsart}
\usepackage{amsmath}
\usepackage{a4wide}
\usepackage[utf8]{inputenc}
\usepackage{amssymb}
\usepackage{amsopn}
\usepackage{epsfig}
\usepackage{amsfonts}
\usepackage{latexsym}
\usepackage{amsthm}
\usepackage{enumerate}
\usepackage[UKenglish]{babel}
\usepackage{verbatim}
\usepackage{color}
\usepackage{calrsfs}
\usepackage{pgf}
\usepackage{tikz}

\makeatletter
\@namedef{subjclassname@2010}{%
  \textup{2010} Mathematics Subject Classification}
\makeatother

\DeclareMathAlphabet{\pazocal}{OMS}{zplm}{m}{n}

\newtheorem{theorem}{Theorem}[section]
\newtheorem{lemma}[theorem]{Lemma}
\newtheorem{proposition}[theorem]{Proposition}

\newtheorem{main}{Theorem}

\theoremstyle{definition}
\newtheorem{definition}[theorem]{Definition}

\theoremstyle{remark}
\newtheorem{remark}[theorem]{Remark}

\numberwithin{equation}{section}

\newcommand{\R}{\ensuremath{\mathbb{R}}}
\newcommand{\N}{\ensuremath{\mathbb{N}}}
\renewcommand{\L}{\ensuremath{\pazocal{L}}}

\renewcommand{\c}{ {\mathbf{c}}}

\newcommand{\B}{\mathcal{B}}
\newcommand{\BS}{ {\mathbf{B}}}
\newcommand{\OB}{\overline{\mathcal{B}}}
\renewcommand{\u}{\ensuremath{\pazocal{U}}}
\newcommand{\ub}{\mathcal{U}}
\newcommand{\us}{\mathbf{U}}
\renewcommand{\v}{\ensuremath{\pazocal{V}}}

\newcommand{\vs}{ {\mathbf{V}}}

\newcommand{\ws} { {\mathbf{W}}}
\newcommand{\set}[1]{\left\{#1\right\}}
\newcommand{\la}{\lambda}
\newcommand{\ga}{\gamma}

\newcommand{\ep}{\varepsilon}
\newcommand{\f}{\infty}
\newcommand{\de}{\delta}
\newcommand{\om}{\omega}
\newcommand{\al}{\alpha}
\newcommand{\lle}{\preccurlyeq}
\newcommand{\lge}{\succcurlyeq}

\newcommand{\si}{\sigma}
\newcommand{\La}{\Lambda}
\newcommand{\ra}{\rightarrow}

\begin{document}

\baselineskip=17pt

\title{On the bifurcation set of unique expansions}

\author{Charlene Kalle}
\address[C. Kalle]{Mathematical Institute, University of Leiden, PO Box 9512, 2300 RA Leiden, The Netherlands}
\email{kallecccj@math.leidenuniv.nl}

\author{Derong Kong}
 
\address[D. Kong]{Mathematical Institute, University of Leiden, PO Box 9512, 2300 RA Leiden, The Netherlands}
\curraddr{College of Mathematics and Statistics, Chongqing University, 401331 Chongqing, China}
\email[Corresponding author]{derongkong@126.com}

\author{Wenxia Li}
\address[W. Li]{School  of Mathematical Sciences, Shanghai Key Laboratory of PMMP, East China Normal University, Shanghai 200062,
People's Republic of China}
\email{wxli@math.ecnu.edu.cn}

\author{Fan L\"{u} }
\address[F. L\"u]{Department of Mathematics, Sichuan Normal University, Chengdu 610068, People's Republic of China}
\email{lvfan1123@163.com}

\date{}


\begin{abstract}
 Given a positive integer $M$, for $q\in(1, M+1]$ let ${\pazocal{U}}_q$ be the set of $x\in[0, M/(q-1)]$ having a unique $q$-expansion with  the digit  set $\{0, 1,\ldots, M\}$, and let $\mathbf{U}_q$ be the set of corresponding $q$-expansions. Recently, Komornik et al.  showed  in \cite{Komornik_Kong_Li_2015_1} that the topological entropy function
 $H: q \mapsto h_{top}(\mathbf{U}_q)$ is a Devil's staircase in $(1, M+1]$.

 Let $\mathcal{B}$ be the bifurcation set of $H$ defined by
 \[
 \mathcal{B}=\{q\in(1, M+1]: H(p)\ne H(q)\quad\textrm{for any}\quad p\ne q\}.
 \]
 In this paper we analyze the fractal properties of $\mathcal{B}$, and show that for any $q\in \mathcal{B}$,
 \[
\lim_{\delta\rightarrow 0} \dim_H(\mathcal{B}\cap(q-\delta, q+\delta))=\dim_H\ensuremath\pazocal{U}_q,
 \]
 where $\dim_H$ denotes the Hausdorff dimension.
Moreover, when $q\in\mathcal{B}$ the univoque set $\ensuremath\pazocal{U}_q$ is dimensionally homogeneous, i.e.,
$
 \dim_H(\ensuremath\pazocal{U}_q\cap V)=\dim_H\ensuremath\pazocal{U}_q
$
 for any open set $V$ that intersect  $\ensuremath\pazocal{U}_q$.

 As an application  we obtain a dimensional spectrum result for the set $\mathcal{U}$ containing  all bases   $q\in(1, M+1]$ such that  $1$ admits a unique $q$-expansion. In particular, we  prove that for any $t>1$ we have
\[
 \dim_H(\mathcal{U}\cap(1, t])=\max_{ q\le t}\dim_H\pazocal{U}_q.
\]
We also consider the variations of the sets $\mathcal{U}=\mathcal{U}(M)$ when   $M$ changes.
\end{abstract}


\maketitle

\section{Introduction}\label{sec: Introduction}
Fix a positive integer $M$. For any $q\in(1, M+1]$ each $x\in I_{q,M}:=[0, {M}/{(q-1)}]$ has a $q$-expansion, i.e., there exists a sequence $(x_i)=x_1x_2\ldots$ with each  $x_i\in\set{0,1,\ldots, M}$ such that
\begin{equation}\label{e11}
x=\sum_{i=1}^\f\frac{x_i}{q^i}=:\pi_q((x_i)).
\end{equation}
The sequence $(x_i)$ is called a \emph{$q$-expansion} of $x$.  
If no confusion arises the \emph{alphabet} is always assumed to be  $\set{0,1,\ldots, M}$.

Non-integer base expansions have received a lot of attention since the pioneering papers of R\'{e}nyi  \cite{Renyi_1957}  and Parry  \cite{Parry_1960}. It is well known that for any $q\in(1,M+1)$ Lebesgue almost every $x\in I_{q,M}$ has a continuum of $q$-expansions (cf.~\cite{Sidorov_2003,Dajani_DeVries_2007}).  Moreover, for any $k\in\N\cup\set{\aleph_0}$ there exist $q\in(1,M+1]$ and $x\in I_{q,M}$ such that $x$ has precisely $k$ different $q$-expansions (see  e.g., \cite{Erdos_Joo_1992,Sidorov_2009}). For more information on non-integer base expansions we refer  the reader to the survey paper \cite{Komornik_2011} and the references therein.

In this paper we focus on studying unique $q$-expansions. For $q\in(1, M+1]$ let
\begin{equation*}\label{e12}
\u_q:=\set{x\in I_{q,M}: x\textrm{ has a unique }q\textrm{-expansion}},
\end{equation*}
and let $\us_q=\pi_q^{-1}(\u_q)$ be the set of corresponding $q$-expansions. These sets have been the object of study in many articles and  have a very rich topological structure (see for example \cite{Komornik_Loreti_2007,DeVries_Komornik_2008}). Komornik et al.  studied in \cite{Komornik_Kong_Li_2015_1} the Hausdorff dimension of $\u_q$, and showed that the dimension function $D: q \mapsto \dim_H\u_q$ has a Devil's staircase behavior (see also \cite{Allaart-Kong-18}). Moreover, they showed that  the entropy function
\[H: (1,M+1]\ra[0, \log(M+1)];\qquad q \mapsto h_{top}(\us_q)\]
 is a Devil's staircase (see Lemma \ref{l24} below). Recently, Alcaraz Barrera et al.  investigated in \cite{AlcarazBarrera-Baker-Kong-2016} the dynamical properties of $\u_q$, and determined the maximal intervals on which the entropy function $H$ is constant.

Let $\B$ be the \emph{bifurcation set} of the function $H$ defined by
\begin{equation*}\label{e12}
\B=\set{q\in(1, M+1]: H(p)\ne H(q)\textrm{ for any } p\ne q}.
\end{equation*}
Then $\B$ is the set of bases where the entropy function $H$ is not locally constant. In \cite{AlcarazBarrera-Baker-Kong-2016} Alcaraz Barrera et al.~gave a characterization of $\B$ and showed that $\B$ has full Hausdorff dimension.  In particular,  we have
\begin{equation} \label{e25}
\B=(q_{KL}, M+1]\setminus\bigcup[p_L, p_R],
\end{equation}
where $q_{KL}$ is the \emph{Komornik-Loreti constant} (cf.~\cite{Komornik_Loreti_2002}) and the union on the right hand side  is countable and pairwise disjoint (see Section \ref{sec:preliminary} below for more explanation). 

From \cite{DeVries_Komornik_2008} we know that the univoque set $\u_q$ has a fractal structure and might have isolated points. Our first result states that for $q\in\B$ the univoque set $\u_q$ is \emph{dimensionally homogeneous}, i.e., the local Hausdorff dimension of $\u_q$ equals the full dimension of $\u_q$.
\begin{main}\label{t11}
Let $q\in(q_{KL}, M+1]\setminus\bigcup(p_L, p_R]$. Then   for any open set $V\subseteq\R$ with $\u_q\cap V\ne\emptyset$ we have
\[
 \dim_H(\u_q\cap V)=\dim_H\u_q.
\]
\end{main}
\begin{remark}
\mbox{}

\begin{enumerate}
\item  Note by (\ref{e25}) that $\B\subset(q_{KL},M+1]\setminus\bigcup(p_L, p_R]$. So Theorem \ref{t11} implies that the univoque set $\u_q$ is dimensionally homogeneous for any $q\in\B$.  
\item {In Theorem \ref{th:dimensionally homogeneous} we give a complete characterization of the set  
\[
\set{q\in(1, M+1]: \u_q\textrm{ is dimensionally homogeneous}}.
\] 
It turns out that  the Lebesgue measure of this set
is  positive and strictly  smaller than $M$. }

\end{enumerate}
\end{remark}

Throughout the paper we will use $\overline{A}$ to denote the topological closure of a set $A\subset\R$. {Our second result presents a close relationship between the  bifurcation set $\overline{\B}$ and the univoque sets $\u_q$.}   
\begin{main}\label{t12}
For any $q\in\overline{\B}$ we have
\[
\lim_{\de\ra 0}\dim_H(\overline{\B}\cap(q-\de, q+\de))=\dim_H\u_q.
\]
\end{main}

\begin{remark}\label{r13}\mbox{}

\begin{enumerate}
\item
Since by (\ref{e25}) and (\ref{e24})  the difference between $\B$ and $\OB$ is countable, Theorem \ref{t12} also holds if we replace $\OB$ by $\B$. 

\item {Note that $\dim_H\u_q>0$ for any $q>q_{KL}$ (see Lemma \ref{l24} below). As a consequence of Theorem \ref{t12} it follows that 
\[
q\in\OB\setminus\set{q_{KL}}\quad\Longleftrightarrow\quad \lim_{\de\ra 0}\dim_H(\OB\cap(q-\de, q+\de))=\dim_H\u_q>0.
\]
Recently, Allaart et al.~\cite[Corollary 3]{Allaart-Baker-Kong-17} gave another characterization of $\overline{\B}$, and showed that 
\[
q\in\OB\setminus\set{q_{KL}}\quad\Longleftrightarrow\quad \lim_{\de\ra 0}\dim_H(\ub\cap(q-\de, q+\de))=\dim_H\u_q>0,
\]
where $\ub:=\set{q\in(1,M+1]: 1\in\u_q}$.  
}

\end{enumerate}
\end{remark}

It is well-known that the  {univoque set} $\u_q$ has a close connection with the set $\ub=\ub(M)$ of \emph{univoque bases} $q\in(1, M+1]$ for which $1$ has a unique $q$-expansion with alphabet $\set{0,1,\ldots,M}$. For example, in \cite{DeVries_Komornik_2008} De Vries and Komornik showed that $\u_q$ is closed if and only if $q\notin\overline{\ub}$. The set $\ub$ has many interesting properties itself. Erd\H{o}s et al.  showed in \cite{Erdos_Joo_Komornik_1990} that $\ub$ is an uncountable set of zero Lebesgue measure. Dar\'{o}czy and K\'{a}tai proved in \cite{Darczy_Katai_1995} that the Hausdorff dimension of $\ub$ is 1 (see also \cite{Komornik_Kong_Li_2015_1}). Komornik and Loreti showed  in \cite{Komornik_Loreti_2002} that the smallest element of $\ub$ is  $q_{KL}$.  In \cite{Komornik_Loreti_2007} the same authors studied the topological properties of $\ub$, and showed that its closure $\overline{\ub}$ is a Cantor set. Recently, Kong et al.  proved in \cite{Kong_Li_Lv_Vries2016} that for any $q\in\overline{\ub}$ we have
\begin{equation}\label{eq12}
\dim_H(\overline{\ub}\cap(q-\de, q+\de))>0\qquad \textrm{for any }\de>0.
\end{equation}

On a different note, in \cite{BCIT13} Bonanno et al.~introduced a set
\begin{equation}\label{e12}
\Lambda = \{ x \in [0,1] \, : \, S^k x \le x \textrm{ for all }n\ge 0 \},
\end{equation}
where $S$ is the tent map defined by $S : x \mapsto \min\{2x, 2-2x \}$ and showed that there is a one to one correspondence between the set $\ub(1)$ and the set $\Lambda \backslash \mathbb Q_1$, where $\mathbb Q_1$ is the set of all rationals with odd denominator. This link is based on work by Allouche and Cosnard (see \cite{All83,AC83,AC01}), who related the set $\ub(1)$ to kneading sequences of unimodal maps. The authors of \cite{BCIT13} also explored a relationship   between these sets and the real slice of the boundary of the Mandelbrot set.

\begin{center}
\begin{figure}[h!]
  \centering
  \includegraphics[width=12cm]{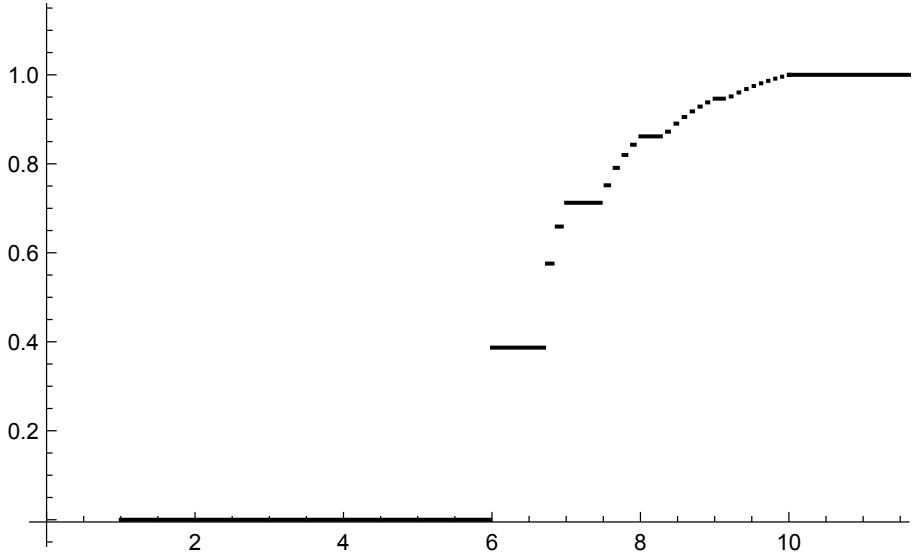}
  \caption{The asymptotic graph  of the function $\phi(t)=\dim_H(\ub\cap(1,t])$ for $t\in[4, 11.5]$ with $M=9$ and $q_{KL}=q_{KL}(9)\approx 5.97592$. }\label{fig:1}
\end{figure}
\end{center}

{By using  Theorem \ref{t12} we investigate the dimensional spectrum of ${\ub}$. Our next result    strengthens   the relationship  between $\u_q$ and $\ub$.}

\begin{main}\label{t14}
For any $t>1$ we have
\[
\dim_H({\ub}\cap(1, t])=\max_{q\le t} \dim_H\u_q.
\]
Moreover, the function $\phi(t):=\dim_H ({\ub}\cap(1, t])$ is a Devil's staircase on $(1, \f)$.
\end{main}

\begin{remark}\label{r15}\mbox{}

\begin{enumerate}
\item In \cite{Komornik_Loreti_2007} it was shown that $\overline{\ub}\setminus\ub$ is a countable set. Hence, Theorem \ref{t14} still holds if we replace ${\ub}$ by $\overline{\ub}$.

\item  Results from \cite{Komornik_Kong_Li_2015_1} (see Lemma \ref{l24} below) give that $\dim_H\u_q=1$ if and only if $q=M+1$. In view of Theorem \ref{t14} we obtain that $\dim_H(\ub\cap(1, t])<1$ for any $t<M+1$. This implies that the Hausdorff dimension of $\ub$ is concentrated on the neighborhood of $M+1$.
\end{enumerate}
\end{remark}

As an application of Theorem \ref{t14} we investigate the variations of $\ub=\ub(M)$ when the parameter $M$ changes. For $K\in\set{1,2,\ldots, M}$, let $\ub(K)$ be the set of bases $q\in(1, K+1]$ such that $1$ has a unique $q$-expansion with respect to  the alphabet $\set{0,1,\ldots, K}$. Theorem \ref{t16} characterizes the Hausdorff dimensions  of the intersection $\ub(M)\cap\ub(K)$
and the difference $\ub(M)\setminus\ub(K)$. Indeed, we prove the following stronger result. 

\begin{main}\label{t16}\mbox{}

\begin{enumerate}[{\rm(i)}]
\item  Let $K\in\set{1,2,\ldots, M}$. Then 
\[\dim_H\left(\bigcap_{J=K}^M\ub(J)\right)=\max_{q\le K+1}\dim_H\u_q.\]

\item  For any positive integer $L$ we have 
\[\dim_H\left(\ub(L)\setminus\bigcup_{J\ne L}\ub(J)\right)=1.\]
\end{enumerate}
\end{main}
\begin{remark}\label{r18}
By the proof of Theorem \ref{t16} it follows  that for $K< M$ the intersection 
\[\bigcap_{J=K}^M\ub(J)=\ub(M)\cap(1, K+1]\]
 is a proper subset of $\ub(K)$. This, together with \eqref{eq12}, implies that for $K<M$ neither the intersection $\bigcap_{J=K}^M\ub(J)$ nor the difference set $\ub(M)\setminus\bigcap_{J=K}^M\ub(J)$ contains isolated points. \end{remark}

{We emphasize that for each $q\in(1, M+1]$ the  univoque set $\u_q$  is related to  the   dynamical system   
\[T_{q,j}: \left[0, \frac{M}{q-1}\right]\ra  \left[0, \frac{M}{q-1}\right];\qquad x \mapsto q x-j \]
for $j\in\set{0,1,\ldots, M}$. On the other hand,  the set $\ub$ contains all   parameters $q\in(1,M+1]$ such that $1$ has a unique $q$-expansion,  and thus  $\ub$ is related  to   infinitely many dynamical systems. }
A similar set up involving a bifurcation set for infinitely many dynamical systems is considered in \cite{BCIT13} (see also \cite{CT13}). They considered  the bifurcation set of an entropy map for a family of maps $\{ T_\alpha: [\alpha -1, \alpha] \to [\alpha -1, \alpha] \}_{\alpha \in [0,1]}$, called $\alpha$-continued fraction transformations \cite{Nakada-1981}, where for each $\al\in[0,1]$ the map $T_\al$ is defined by 
\begin{equation}\label{eq:k1}
T_\al(x)=\left\{\begin{array}{cll}
\frac{1}{|x|}-\lfloor\frac{1}{|x|}+1-\al\rfloor &\textrm{if}&  x\ne 0;\\
0&\textrm{if}& x=0.
\end{array}\right.
\end{equation}
Each map $T_\al$ has  a unique invariant measure $\mu_\alpha$ that is absolutely continuous with respect to the Lebesgue measure. They showed that    the map 
\[\psi: \alpha \mapsto h_{\mu_\alpha}(T_\alpha),\] assigning to each $\alpha$ the measure theoretic  entropy $h_{\mu_{\alpha}}(T_\alpha)$, has countably many intervals on which it is monotonic.  The complement of the union of  these intervals in $[0,1]$, i.e., the bifurcation set of $\psi$ denoted by $F$, has Lebesgue measure 0 (see \cite{KSS12} and \cite{CT13}) and    Hausdorff dimension 1 (see  \cite{BCIT13}). Moreover, in \cite{BCIT13} the authors identified a homeomorphism between the set $F$ and the set $\Lambda\setminus\set{0}$ from (\ref{e12}), giving also a relation to the set $\ub(1)$. In \cite{BCIT13}, however, no information is given on the local structure of $F$. 
Recently, Dajani and the first author identified in \cite{DK}   another set $E$ that is linked to the sets $\ub(1)$, $\Lambda$ and $F$. They investigated a family of symmetric doubling maps $S_\gamma: [-1,1] \to [-1,1]$, given by 
\[S_\gamma (x) = 2x- \gamma \lfloor 2x \rfloor,\]
 where $\lfloor x \rfloor$ denotes the integer part of $x$, and showd that the set $E$ of parameters $\gamma \in [1,2]$ for which the map $S_\gamma$ does not have a piecewise smooth invariant density is homeomorphic to $\Lambda\setminus\set{0}$. { Therefore, the results obtained in this paper about the set $\ub(1)$ can be
used to investigate  the bifurcation sets $E$, $F$ and the set $\Lambda$.}

The rest of the paper is arranged in the following way. In Section \ref{sec:preliminary} we fix the notation and recall some properties of unique $q$-expansions. Moreover, we recall from \cite{AlcarazBarrera-Baker-Kong-2016} some important properties of the bifurcation set $\B$. In Section \ref{sec:proof of thm 2} we give the proof of Theorem \ref{t11}  for the dimensional homogeneousness of $\u_q$.  In Section \ref{sec:aux prop} we prove an auxiliary proposition that will be used to prove Theorem \ref{t12} in  Section \ref{sec:proof of thm1}. The proof  of Theorems \ref{t14} and \ref{t16} will be given in Sections \ref{sec:proof of th1} and \ref{sec:proof of th2}, respectively. We end the paper with some   remarks.

\section{Unique expansions and bifurcation set}\label{sec:preliminary}
In this section we   recall some  properties of unique $q$-expansions and of the bifurcation set $\B$ as well.
First we need some terminology from symbolic dynamics (cf.~\cite{Lind_Marcus_1995}).

\subsection{Symbolic dynamics}
Given a positive integer $M$, let $\{ 0,1, \ldots, M\}^*$ denote the set of all finite strings of symbols from $\{ 0,1, \ldots, M\}$, called {\em words}, together with the empty word denoted by $\epsilon$. Let $\set{0,1,\ldots, M}^\N$ be the set of sequences $(d_i)=d_1d_2\ldots$ with each element $d_i\in\set{0,1,\ldots, M}$. Let $\si$ be the left shift on $\set{0,1,\ldots, M}^\N$ defined by $\si((d_i))=(d_{i+1})$. Then $(\set{0,1,\ldots, M}^\N, \si)$ is a full shift. For a  {word} $\c=c_1\ldots c_n \in \{0,1, \ldots, M\}^*$ we denote by $\c^k=(c_1\ldots c_n)^k$ the $k$-fold concatenation of $\c$ to itself and by $\c^\f=(c_1\ldots c_n)^\f$ the periodic sequence with period block $\c$. Moreover, for a word $\c=c_1\ldots c_n$ with $c_n<M$ we denote by $\c^+$ the word
\[\c^+=c_1\ldots c_{n-1}(c_n+1).\]
 Similarly, for a word $\c=c_1\ldots c_n$ with $c_n>0$ we write $\c^-=c_1\ldots c_{n-1}(c_n-1)$.
 For a sequence $(d_i)\in\set{0,1,\ldots, M}^\N$ we denote its \emph{reflection} by
\[
\overline{(d_i)}=(M-d_1)(M-d_2)\cdots.
\]
Accordingly, for a word $\c=c_1\ldots c_n$ we denote its reflection by $\overline{\c}=(M-c_1)\cdots (M-c_n)$.

On words and sequences we consider the lexicographical ordering $\prec, \lle, \succ$ or $\lge$ which is defined as follows. For two sequences $(c_i), (d_i)\in\set{0,1,\ldots, M}^\N$ we say that $(c_i)\prec (d_i)$ if there exists  $n\in\N$ such that $c_1\ldots c_{n-1}=d_1\ldots d_{n-1}$ and $c_n<d_n$. Moreover, we write $(c_i)\lle (d_i)$ if $(c_i)\prec (d_i)$ or $(c_i)=(d_i)$. Similarly, we write $(c_i)\succ (d_i)$ if $(d_i)\prec (c_i)$, and   $(c_i)\lge (d_i)$ if $(d_i)\lle(c_i)$. We extend this definition to words in the following way. For two words $\om, \nu \in \{0,1,\ldots, M\}^*$ we write $\om\prec\nu$ if $\om 0^\f\prec\nu0^\f$. Accordingly, for a sequence $(d_i)\in\set{0, 1,\ldots, M}^\N$ and a word $\c=c_1\ldots c_m$ we say $(d_i)\prec \c$ if $(d_i)\prec\c0^\f$.

Let $\pazocal{F} \subseteq \set{0,1,\ldots,M}^*$ and let $X=X_{\pazocal F} \subseteq \set{0,1,\ldots,M}^\N$ be the set of those sequences that do not contain any word from $\pazocal{F}$. We call the pair $(X, \si)$ a \emph{subshift}. If $\pazocal{F}$ can be chosen to be a finite set, then $(X, \si)$ is called a \emph{subshift of finite type}. For $n\in\N\cup\set{0}$ we denote by $\L_n(X)$ the set of words of length $n$ occurring in  sequences of $X$. In particular, for $n=0$ we set $\L_0(X)=\set{\epsilon}$. The \emph{languange} of $(X, \si)$ is then defined by
\[
\L(X)=\bigcup_{n=0}^\f \L_n(X).
\]
So, $\L(X)$ is the set of all finite words occurring in sequences from $X$.

For a subshift $(X, \si)$ and a word $\om\in\L(X)$ let $F_X(\om)$ be the \emph{follower set} of $\om$ in $X$ defined by
\begin{equation}\label{e21}
F_X(\om):=\set{(d_i)\in X: d_1\ldots d_{|\om|}=\om},
\end{equation}
where $|\c|$ denotes the length of a word $\c\in\set{0,1,\ldots,M}^*$.

A subshift $(X, \si)$ is called \emph{topologically transitive} (or simply \emph{transitive}) if for any two words $\om, \nu\in\L(X)$ there exists a word $\gamma$ such that $\om\gamma\nu\in\L(X)$. In other words, in a transitive subshift $(X, \si)$ any two words can be ``connected" in $\L(X)$.

The \emph{topological entropy} $h_{top}(X)$ of a subshift $(X, \si)$ is a quantity that indicates its complexity. It is defined by
\begin{equation}\label{e22}
h_{top}(X)=\lim_{n\ra\f}\frac{\log \#\L_n(X)}{n}=\inf_{n\ge 1}\frac{\log \# \L_n(X)}{n},
\end{equation}
where $\# A$ denotes the cardinality of a set $A$. Accordingly, we define the topological entropy of a follower set $F_X(\om)$ by changing $X$ to $F_X(\om)$  in \eqref{e22} if the corresponding limit exists.
Clearly, if $X$ is a transitive subshift, then $h_{top}(F_X(\om))=h_{top}(X)$ for any $\om\in\L(X)$.

\subsection{Unique expansions}
In this subsection we recall some results about unique expansions. For more information on this topic we refer the reader to the survey papers \cite{Sidorov_2003_survey,Komornik_2011} or the book chapter \cite{deVries-Komornik-2016}. For $q\in(1, M+1]$, let
\[\al(q)=\al_1(q)\al_2(q)\ldots\]
be the \emph{quasi-greedy} $q$-expansion of $1$ (cf.~\cite{Daroczy_Katai_1993}), i.e., the lexicographically largest $q$-expansion of $1$ not ending with a string of zeros. The following characterization of quasi-greedy expansions was given in \cite[Theorem 2.2]{Baiocchi_Komornik_2007}.
 \begin{lemma}
 \label{l21}
 The map $q\mapsto \al(q)$ is a strictly increasing bijection from $(1, M+1]$ onto the set of all sequences $(a_i)\in\set{0,1,\ldots, M}^\N$ not ending with $0^\f$ and satisfying
 \[
 a_{n+1}a_{n+2}\ldots \preceq a_1a_2\ldots\qquad\textrm{whenever}\quad a_n<M.
 \]
 \end{lemma}

Recall from (\ref{e11}) the definition of the projection map $\pi_q$ for $q\in(1,M+1]$ mapping $\set{0, 1,\ldots, M}^\N$ onto the interval $I_{q, M}=[0, M/(q-1)]$.  In general, $\pi_q$ is not bijective. However,  $\pi_q$ is a bijection between $\us_q=\pi_q^{-1}(\u_q)$ and $\u_q$. The following lexicographical characterization of $\us_q$, or equivalently $\u_q$, was essentially due to Parry  \cite{Parry_1960} (see also \cite{Baiocchi_Komornik_2007}).
\begin{lemma}\mbox{}\label{l22}
Let $q\in(1, M+1]$. Then $(x_i)\in\us_q$ if and only if
\begin{align*}
x_{n+1}x_{n+2}\ldots\prec \al(q) &\qquad \textrm{whenever}\quad x_n<M,\\
\overline{x_{n+1}x_{n+2}\ldots}\prec \al(q)&\qquad \textrm{whenever}\quad x_n>0.
\end{align*}
\end{lemma}
Observe that $\ub=\set{q\in(1, M+1]: \al(q)\in\us_q}$. 
As a corollary of Lemma \ref{l22} we have the following characterizations of $\ub$ and $\overline{\ub}$.
\begin{lemma}\label{l23}\mbox{}
\begin{enumerate}[{\rm(i)}]
  \item $q\in\ub\setminus\set{M+1}$ if and only if the quasi-greedy expansion $\al(q)$ satisfies
\[
\overline{\al(q)}\prec \si^n(\al(q))\prec \al(q)\qquad\textrm{for any }n\ge 1.
\]
  \item $q\in\overline{\ub}$ if and only if the quasi-greedy expansion $\al(q)$ satisfies
  \[
\overline{\al(q)}\prec \si^n(\al(q))\lle  \al(q)\qquad\textrm{for any }n\ge 1.
\]
\end{enumerate}
\end{lemma}
\begin{proof}
Part (i) was shown in \cite[Theorem 2.5]{Vries-Komornik-Loreti-2016} and Part (ii) was proven in \cite[Theorem 3.9]{Vries-Komornik-Loreti-2016}.
\end{proof}

In \cite{DeVries_Komornik_2008} it was shown that $(\us_q, \si)$ is not necessarily a subshift. Inspired by \cite{Komornik_Kong_Li_2015_1} we consider the set $\vs_q$ which contains all sequences $(x_i)\in\set{0,1,\ldots, M}^\N$ satisfying
\[
\overline{\al(q)}\lle\si^n((x_i))\lle \al(q)\qquad\textrm{for all}\quad n\ge 0.
\]
Then $(\vs_q, \si)$ is a subshift (cf.~\cite[Lemma 2.6]{Komornik_Kong_Li_2015_1}). Furthermore, Lemma \ref{l21} implies that the set-valued map $q\mapsto \vs_q$ is increasing, i.e., $\vs_p\subseteq\vs_q$ whenever $p<q$.  

Recall that the Komornik-Loreti constant $q_{KL}$ is the smallest element of $\ub$, which is defined in terms of the classical  \emph{Thue-Morse sequence} $(\tau_i)_{i=0}^\f=01101001\ldots$. Here the sequence $(\tau_i)_{i=0}^\f$ is defined as follows (cf.~\cite{Allouche_Shallit_1999}): $\tau_0=0$, and if $\tau_0\ldots\tau_{2^n-1}$ has already been defined for some $n\ge 0$, then $\tau_{2^n}\ldots\tau_{2^{n+1}-1}=\overline{\tau_0\ldots\tau_{2^n-1}}$. Then the Komornik-Loreti constant $q_{KL}=q_{KL}(M)\in(1, M+1]$ is the unique base satisfying 
\begin{equation}\label{e23}
\al(q_{KL}) =\la_1\la_2\ldots,
\end{equation}
where
\[
\la_i=\left\{
\begin{array}{lll}
  k+\tau_i-\tau_{i-1} & \textrm{if} & M=2k,  \\
  k+\tau_i& \textrm{if} & M=2k+1,
\end{array}\right.
\]
for each $i\ge 1$. We emphasize  that the sequence $(\la_i)$ depends on $M$. By the definition  of the Thue-Morse sequence $(\tau_i)_{i=0}^\f$ it follows that  (cf.~\cite{AlcarazBarrera-Baker-Kong-2016})
\begin{equation}\label{eq:lambda}
\la_{2^n+1}\ldots\la_{2^{n+1}}=\overline{\la_1\ldots\la_{2^n}}\,^+\qquad\textrm{for any}\quad n\ge 0.
\end{equation}

Recall that a function $f: [a, b]\ra \R$ is called a \emph{Devil's staircase} (or \emph{Cantor function}) if $f$ is a continuous and non-decreasing function with $f(a)<f(b)$, and $f$ is locally constant   almost everywhere. 
 The next lemma   summarizes  some results from \cite{Komornik_Kong_Li_2015_1} on the Hausdorff dimension of $\u_q$.
\begin{lemma}\label{l24}\mbox{}
\begin{enumerate}[{\rm(i)}]

  \item  For any $q\in(1, M+1]$ we have
\[
\dim_H\u_q=\frac{h_{top}(\vs_q)}{\log q}.
\]

  \item The entropy function $H: q\mapsto h_{top}(\vs_q)$ is a Devil's staircase in $(1, M+1]$:
\begin{itemize}
\item $H$ is increasing and  continuous  in $(1, M+1]$;
\item $H$ is locally constant  almost everywhere in $(1, M+1]$;
\item $H(q)=0$ if and only if $1<q\le q_{KL}$. Moreover,   $H(q)=\log(M+1)$ if and only if $q=M+1$.
\end{itemize}
\end{enumerate}
\end{lemma}

\begin{remark}\label{r25}\mbox{}

\begin{enumerate}
\item Lemma \ref{l24} implies that the dimensional function $D: q\mapsto \dim_H\u_q$ has a Devil's staircase behavior: (i) $D$ is continous in $(1, M+1]$; (ii) $D'<0$ almost everywhere in $(1, M+1]$; (iii) $D(q)=0$ for any $q\in(1, q_{KL}]$ and $D(q)=1$ for $q=M+1$. 

\item
In \cite[Lemma 2.11]{Komornik_Kong_Li_2015_1} the authors showed that $H$ is locally constant on the complement of $\overline{\ub}$, i.e., $H'(q)=0$ for any $q\in(1, M+1]\setminus\overline{\ub}$.
\end{enumerate}
\end{remark}

\subsection{Bifurcation set} In this subsection we recall some recent results obtained in \cite{AlcarazBarrera-Baker-Kong-2016}, where the authors investigated the maximal intervals on which $H$ is locally constant, called \emph{entropy plateaus} (or simply called \emph{plateaus}). For convenience of the reader we adopt much of the notation  from \cite{AlcarazBarrera-Baker-Kong-2016}. We hope that this helps the interested reader who wants to access the relevant background information. Let $\B$ be the complement of these plateaus. From  Lemma \ref{l24} (ii) we have
\[
\B=\set{q\in(1, M+1]: H(p)\ne H(q) \quad\textrm{for any }p\ne q}.
\]
Note by (\ref{e25}) that $\B$ is not closed. For the closure $\OB$ we have
\[
\OB=\set{q\in(1,M+1]: \forall \de>0, \exists p\in(q-\de, q+\de)\textrm{ such that }H(p)\ne H(q)}.
\]
In \cite{AlcarazBarrera-Baker-Kong-2016}   $\OB$ was denoted by $\mathcal E$. The following lemma for $\OB$, the first part of which follows from Remark \ref{r25} (2), was established in \cite[Theorem 3]{AlcarazBarrera-Baker-Kong-2016}.
\begin{lemma}
  \label{l26}
  $\OB\subset\overline{\ub}$, and $\OB$ is a Cantor set of full Hausdorff dimension.
\end{lemma}

By Lemma \ref{l24}   it follows that $\min\OB=q_{KL}$ and $\max\OB=M+1$. Since $\OB$ is a Cantor set, we can write
\begin{equation}\label{e24}
  (q_{KL}, M+1]\setminus\OB=\bigcup(p_L, p_R),
\end{equation}
where the union is pairwise disjoint and countable. By the definition of $\OB$ it follows that the intervals $[p_L,p_R]$ are the plateaus of $H$.  In particular, since $H$ is increasing,  these intervals have the property that $H(q)=H(p_L)$ if and only if $q\in[p_L, p_R]$. This implies that the bifurcation set $\B$ can be rewritten as in (\ref{e25}), i.e., 
\begin{equation*} 
  \B=(q_{KL}, M+1]\setminus\bigcup[p_L, p_R].
\end{equation*}
By (\ref{e24}) and (\ref{e25}) it follows that $\OB\setminus\B$ is countable. 
The fact that $\OB$ does not have isolated points gives the following lemma (see also \cite{AlcarazBarrera-Baker-Kong-2016}).

\begin{lemma}\label{l27}\mbox{}

\begin{enumerate}[{\rm(i)}]
\item For any $q\in(q_{KL}, M+1]\setminus\bigcup(p_L, p_R]$ there exists a sequence of plateaus $\set{[p_L(n), p_R(n)]}$ such that $p_L(n)\nearrow q$ as $n\ra\f$. 

\item For any $q\in[q_{KL}, M+1)\setminus\bigcup[p_L, p_R)$ there exists a sequence of plateaus $\set{[q_L(n), q_R(n)]}$ such that $q_L(n)\searrow q$ as $n\ra\f$.
\end{enumerate}
\end{lemma}

So, by \eqref{e24}, \eqref{e25} and Lemma \ref{l27} it follows that $\OB\setminus\B$ is a countable and dense subset of $\OB$. In particular, the set
of left endpoints of all plateaus of $H$ is dense in $\OB$.

In \cite{AlcarazBarrera-Baker-Kong-2016} more detailed information on the structure of the plateaus of $H$ is given. Before we can give the necessary details, we have to recall some notation from \cite{AlcarazBarrera-Baker-Kong-2016}. Let $\vs$ be the set of sequences $(a_i)\in\set{0,1,\ldots, M}^\N$ satisfying the inequalities
\[
\overline{(a_i)}\lle \si^n((a_i))\lle (a_i)\qquad\textrm{for all}\quad n\ge 0.
\]
In \cite[Lemma 3.3]{AlcarazBarrera-Baker-Kong-2016} it is proved that the subshift $(\vs_q, \si)$ is not transitive for any $q\in(q_{KL}, q_T)$, where $q_T\in(1, M+1)\cap\B$ is the unique base such that 
\begin{equation}\label{e26}
\al(q_T)=\left\{\begin{array}{lll}
                  (k+1)k^\f & \textrm{if} & M=2k,  \\
                  (k+1)((k+1)k)^\f & \textrm{if} & M=2k+1.
                \end{array}\right.
\end{equation}
The plateaus of $H$ are characterized separately for the cases (A) $q\in[q_T, M+1]$ and (B) $q\in(q_{KL}, q_T). $

(A). First we recall from \cite{AlcarazBarrera-Baker-Kong-2016}  the following definition.
\begin{definition}\label{def:irreducible}
A sequence $(a_i)\in \vs$ is called \emph{irreducible} if
\[
a_1\ldots a_j(\overline{a_1\ldots a_j}\,^+)^\f\prec (a_i)\qquad\textrm{whenever}\quad (a_1\ldots a_j^-)^\f\in\vs.
\]
\end{definition}

\begin{lemma}\label{l28}\mbox{}
  Let $[p_L, p_R]\subset[q_T, M+1]$ be a plateau of $H$.
  \begin{enumerate}[{\rm(i)}]
    \item There exists a word $a_1\ldots a_m\in\L(\vs_{p_L})$ such that 
    \[\al(p_L)=(a_1\ldots a_m)^\f\textrm{ is irreducible,}\quad\textrm{and} \quad \al(p_R)=a_1\ldots a_m^+(\overline{a_1\ldots a_m})^\f.\]
    \item $(\vs_{p_L},\si)$ is a transitive subshift of  finite type.
    \item There exists a periodic sequence $\nu^\f\in\vs_{p_L}$ such that for any  word  $\eta\in\L(\vs_{p_L})$ we can find a large integer $N$ and    a word $\om$ satisfying
        \[
    \overline{\al_1(p_L)\ldots\al_N(p_L)}\prec\si^j(\eta\om\nu^\f)\prec\al_1(p_L)\ldots\al_N(p_L) \qquad\textrm{for any}\quad  j\ge 0.
    \]
  \end{enumerate}
\end{lemma}
\begin{proof}
Part (i) follows by \cite[Proposition 5.2]{AlcarazBarrera-Baker-Kong-2016}, and  Part (ii) follows by \cite[Lemma 5.1 (1)]{AlcarazBarrera-Baker-Kong-2016}.

For (iii) we take
\[
\nu=\left\{
\begin{array}{lll}
k&\textrm{if}& M=2k,\\
(k+1)k&\textrm{if}& M=2k+1.
\end{array}\right.
\]
  Since $p_L\ge q_T$,  by Lemma \ref{l21} we have $\al(p_L)\lge\al(q_T)$. Then  \eqref{e26}   gives that
\begin{equation}\label{e27}
\overline{\al_1(p_L) \al_2(p_L)}\lle \overline{\al_1(q_T)\al_2(q_T)}\prec \si^j(\nu^\f)\prec\al_1(q_T)\al_2(q_T)\lle \al_1(p_L) \al_2(p_L) 
\end{equation}
for all $j\ge 0$.
Note by (i) that $\al(p_L)$ is irreducible. By     the proof of \cite[Proposition 3.17]{AlcarazBarrera-Baker-Kong-2016} it follows that for any word $\eta\in\L(\vs_{p_L})$ there exist a large integer $N\ge 2$ and a word $\om$ satisfying
    \[
    \overline{\al_1(p_L)\ldots\al_N(p_L)}\prec\si^j(\eta\om\nu^\f)\prec\al_1(p_L)\ldots\al_N(p_L) \quad\textrm{for any}\quad  0\le j<|\eta|+|\om|.
    \]
    This together with \eqref{e27} proves (iii).
\end{proof}

(B). Now we consider plateaus of $H$ in $(q_{KL}, q_T)$. Let $(\la_i)$ be the quasi-greedy $q_{KL}$-expansion of $1$ as given in \eqref{e23}. Note that $(\la_i)$ depends on $M$. For $n\ge 1$ let
\begin{equation}\label{e28}
\xi(n)=\left\{\begin{array}{lll}
                \la_1\ldots \la_{2^{n-1}}(\overline{\la_1\ldots \la_{2^{n-1}}}\,^+)^\f & \textrm{if} & M=2k, \\
                 \la_1\ldots \la_{2^{n }}(\overline{\la_1\ldots \la_{2^{n }}}\,^+)^\f  & \textrm{if} & M=2k+1.
              \end{array}\right.
\end{equation}
Then $\xi(1)=\al(q_T)$, and $\xi(n)$ is strictly decreasing to $(\la_i)=\al(q_{KL})$ as $n\ra\f$. Moreover,  \cite[Lemma 4.2]{AlcarazBarrera-Baker-Kong-2016}   gives that  $\xi(n)\in\vs$ for all $n\ge 1$. We recall from \cite{AlcarazBarrera-Baker-Kong-2016} the following definition.

\begin{definition}\label{def:*-irreducible}
A sequence $(a_i)\in \vs$ is said to be \emph{$*$-irreducible} if there exists $n\in\N$
 such that $\xi(n+1)\lle (a_i)\prec \xi(n)$, and
 \[
 a_1\ldots a_j(\overline{a_1\ldots a_j}\,^+)^\f\prec (a_i)
 \]
 whenever
 \[
 (a_1\ldots a_j^-)^\f\in\vs\quad\textrm{and}\quad j>\left\{\begin{array}{lll}
                                                             2^n & \textrm{if} & M=2k, \\
                                                             2^{n+1} & \textrm{if} & M=2k+1.
                                                           \end{array}\right.
 \]
\end{definition}

\begin{lemma}
\label{l29}
Let $[p_L, p_R]\subseteq(q_{KL}, q_T)$ be a plateau of $H$.
\begin{enumerate}[{\rm(i)}]
\item There exists a word $a_1\ldots a_m\in\L(\vs_{p_L})$ such that 
\[\al(p_L)=(a_1\ldots a_m)^\f\textrm{ is $*$-irreducible}, \quad \textrm{and}\quad  \al(p_R)=a_1\ldots a_m^+(\overline{a_1\ldots a_m})^\f.\]

\item $(\vs_{p_L}, \si)$ is a subshift of finite type, and it  contains a unique transitive subshift of finite type $(X_{p_L}, \si)$ satisfying 
$h_{top}(X_{p_L})=h_{top}(\vs_{p_L}).$

\item There exists a periodic sequence $\nu^\f\in X_{p_L}$ such that  for any word $\eta\in\L(\vs_{p_L})$ we can find    a large integer $N$ and a word $\om$ satisfying
\[
\overline{\al_1(p_L)\ldots\al_N(p_L)}\prec\si^j(\eta\om\nu^\f)\prec\al_1(p_L)\ldots\al_N(p_L) \quad\textrm{for any}\quad j\ge 0.
\]

\end{enumerate}
\end{lemma}
\begin{proof}
Part (i) follows   from \cite[Proposition 5.11]{AlcarazBarrera-Baker-Kong-2016}, and Part (ii) follows from \cite[Lemma 5.9]{AlcarazBarrera-Baker-Kong-2016}. Then it remains to prove (iii).

By (i) we know that $\al(p_L)$ is a $*$-irreducible sequence. Then there exists  $n\in\N$ such that $\xi(n+1)\lle\al(p_L)\prec\xi(n)$. Note by (i) and (\ref{e28}) that   $\al(p_L)$ is purely periodic while $\xi(n+1)$ is eventually  periodic. Then   $\al(p_L)\succ \xi(n+1)$. Let
\[
\nu=\left\{
\begin{array}{lll}
\la_1\ldots\la_{2^n}^-&\textrm{if}& M=2k,\\
\la_1\ldots\la_{2^{n+1}}^-&\textrm{if}& M=2k+1.
\end{array}\right.
\]
Then by the proof of \cite[Lemma 5.9]{AlcarazBarrera-Baker-Kong-2016} we have $\nu^\f\in X_{p_L}$. Observe by \eqref{eq:lambda} and \eqref{e28} that $\xi(n+1)=\nu^+(\overline{\nu})^\f\in\vs$. Then by using $\al(p_L)\succ\xi(n+1)$ it follows that there exists a large integer $N$ such that
\[
\overline{\al_1(p_L)\ldots\al_N(p_L)}\prec\si^j(\nu^\f)\prec\al_1(p_L)\ldots\al_N(p_L)\quad\textrm{for any}\quad j\ge 0.
\]
The remaining  part of  (iii) follows by  the proof of \cite[Lemma 5.8]{AlcarazBarrera-Baker-Kong-2016}.
\end{proof}

Finally, the following characterization of $\OB$ was established in \cite[Theorem 3]{AlcarazBarrera-Baker-Kong-2016}.
\begin{lemma}
\label{l220}
\mbox{}
$
\OB=\overline{\set{q\in(q_{KL}, M+1]: \al(q)\textrm{  is irreducible or }*-\textrm{irreducible}}}.
$
\end{lemma}

\section{Dimensional homogeneity of $\u_q$}\label{sec:proof of thm 2}

In this section we will prove Theorem   \ref{t11}.  Instead of proving  Theorem \ref{t11} we  prove the following equivalent statement.

\begin{theorem}\label{t31}
Let $q\in(1,q_{KL}]\cup((q_{KL}, M+1]\setminus\bigcup(p_L, p_R])$.
Then for any $x\in\u_q$ we have
\begin{equation*}
\dim_H( \u_q\cap(x-\de, x+\de))=\dim_H\u_q\qquad\textrm{for any }\de>0.
\end{equation*}
\end{theorem}

Before giving the proof of Theorem \ref{t31} we first explain  why Theorem \ref{t31} is equivalent to  Theorem \ref{t11}. Clearly, Theorem \ref{t11} implies Theorem     \ref{t31}. On the other hand, take $q\in\B$. Let $V\subseteq\R$ be an open set with $\u_q\cap V\ne\emptyset$. Then there exist  $x\in\u_q\cap V$ and  $\de>0$ such that
\[
 \u_q\cap V\supset\u_q\cap(x-\de, x+\de).
\]
By Theorem \ref{t31} it follows that $\dim_H(\u_q\cap V)\ge \dim_H\u_q$, which gives Theorem \ref{t11}.

Note that for $q \in (1,q_{KL}]$ the statement of Theorem \ref{t31} follows immediately from the fact that $\dim_H \u_q =0$. For $q\in(q_{KL}, M+1]$ recall that $\vs_q$ is the set of sequences $(x_i)\in\set{0,1,\ldots, M}^\N$ satisfying
 \[\overline{\al(q)}\lle \si^n((x_i))\lle \al(q)\quad\textrm{for all }n\ge 0.\]
Accordingly, let 
\[\v_q:=\set{\pi_q((x_i)): (x_i)\in\vs_q},\]
 where $\pi_q$ is the projection map defined in (\ref{e11}). For a set $A\subset\R$ and   $r\in\R$ we denote by $rA:=\set{r\cdot a: a\in A}$ and $r+A:=\set{r+a: a\in A}$.

The following lemma for a  relationship between $\u_q$ and $\v_q$ follows   from Lemma \ref{l22} and the definition of $\v_q$.
\begin{lemma}\label{l32}
Let $q\in(q_{KL}, M+1]$. Then $\u_q$ is  a countable union of affine  copies of $\v_q$ up to a countable set, i.e.,
\begin{align*}
\u_q\cup\pazocal{N}=&\set{0,\frac{M}{q-1}}\cup \bigcup_{c_1=1}^{M-1}\left(\frac{c_1}{q}+\frac{\v_q}{q}\right)\cup\bigcup_{m=2}^\f\bigcup_{c_m=1}^M\left(\frac{c_m}{q^m}+\frac{\v_q}{q^m}\right)\\
&\quad\cup\bigcup_{m=2}^\f\bigcup_{c_m=0}^{M-1}\left(\sum_{i=1}^{m-1}\frac{M}{q^i}+\frac{c_m}{q^m}+\frac{\v_q}{q^m}\right),
\end{align*}
where  the set $\pazocal{N}$ is   at most countable.
\end{lemma}
\begin{proof}
  For $q\in(q_{KL}, M+1]$ let $\mathbf{W}_q$ be the set of sequences $(x_i)$ satisfying  
  \[
  \overline{\al(q)}\prec\si^n((x_i))\prec \al(q)\quad\textrm{for any}\quad n\ge 0,
  \]
  and let $\pazocal{W}_q=\pi_q(\mathbf{W}_q)$.
  Then $\v_q\setminus\pazocal{W}_q$ is at most countable (cf.~\cite{DeVries_Komornik_2008}).    By \cite[Lemma 2.5]{Komornik_Kong_Li_2015_1} it follows that 
  \begin{align*}
\u_q=&\set{0,\frac{M}{q-1}}\cup \bigcup_{c_1=1}^{M-1}\left(\frac{c_1}{q}+\frac{\pazocal{W}_q}{q}\right)\cup\bigcup_{m=2}^\f\bigcup_{c_m=1}^M\left(\frac{c_m}{q^m}+\frac{\pazocal{W}_q}{q^m}\right)\\
&\quad\cup\bigcup_{m=2}^\f\bigcup_{c_m=0}^{M-1}\left(\sum_{i=1}^{m-1}\frac{M}{q^i}+\frac{c_m}{q^m}+\frac{\pazocal{W}_q}{q^m}\right).
\end{align*}
This establishes the lemma  since $\pazocal{W}_q\subseteq\v_q$ and $\v_q\setminus\pazocal{W}_q$ is at most countable.  
  \end{proof}

It immediately follows from Lemma \ref{l32} that
  \[\dim_H\u_q=\dim_H\v_q\qquad\textrm{for any }q\in(q_{KL}, M+1].\]
Hence, it suffices to prove Theorem  \ref{t31} for $\v_q$ instead of $\u_q$. We first prove it for $q$ being the left endpoint of an entropy plateau.
\begin{lemma}\label{l33}
Let $[p_L, p_R]\subset(q_{KL}, M+1)$ be a plateau of $H$. Then for any $x\in\v_{p_L}$ we have
\[
\dim_H(\v_{p_L}\cap(x-\de, x+\de))=\dim_H\v_{p_L}\qquad\textrm{for any }~\de>0.
\]
\end{lemma}
\begin{proof}
Obviously,  $\dim_H(\v_{p_L}\cap(x-\de, x+\de))\le \dim_H\v_{p_L}$.
So, it suffices to the prove the reverse inequality.

Fix $\de>0$ and  $x\in\v_{p_L}$. Suppose that $(x_i)\in\vs_{p_L}$ is a $p_L$-expansion of $x$. Then there exists a large integer $N$ such that
\begin{equation}\label{e31}
\pi_{p_L}(F_{\vs_{p_L}}(x_1\ldots x_N))\subseteq \v_{p_L}\cap(x-\de, x+\de),
\end{equation}
where the follower set $F_{\vs_{p_L}}(x_1\ldots x_N)=\{ (y_i)\in\vs_{p_L}: y_1\ldots y_N=x_1\ldots x_N \}$ is as defined in \eqref{e21}. We split the proof into the following two cases.

Case I. $[p_L, p_R]\subset[q_T, M+1]$. Then by Lemma \ref{l28} (ii) it follows that $(\vs_{p_L}, \si)$ is a transitive subshift of finite type. This implies that
\[
h_{top}(F_{\vs_{p_L}}(x_1\ldots x_N))=h_{top}(\vs_{p_L}).
\]
Then, by \eqref{e31}, Lemma \ref{l24} (i) and Lemma \ref{l32}  it follows that
\begin{align*}
\dim_H(\v_{p_L}\cap(x-\de, x+\de))&\ge\dim_H\pi_{p_L}(F_{\vs_{p_L}}(x_1\ldots x_N))\\
&=\frac{h_{top}(F_{\vs_{p_L}}(x_1\ldots x_N))}{\log p_L}\\
&=\frac{h_{top}(\vs_{p_L})}{\log p_L}=\dim_H\u_{p_L}=\dim_H\v_{p_L}.
\end{align*}

Case II. $[p_L, p_R]\subset(q_{KL}, q_T)$. Then by Lemma \ref{l29} (ii) it follows that $(\vs_{p_L}, \si)$ is a subshift of finite type that contains a unique transitive subshift of finite type $X_{p_L}$ such that
\begin{equation}\label{e32}
h_{top}(X_{p_L})=h_{top}(\vs_{p_L}).
\end{equation}
Furthermore, by Lemma \ref{l29} (iii) there exist  a sequence  $\nu^\f\in X_{p_L}$ and a word $\om$ such that
\begin{equation}\label{eq:k3}
x_1\ldots x_N\om\nu^\f\in F_{\vs_{p_L}}(x_1\ldots x_N).
\end{equation}
 From \cite[Proposition 2.1.7]{Lind_Marcus_1995} there exists $m\ge 0$ such that  $(\vs_{p_L}, \si)$ is an $m$-step subshift of finite type.  Note by    \eqref{eq:k3} that the word $x_1\ldots x_N\om\nu^m\in\L(\vs_{p_L})$. Then 
by \cite[Theorem 2.1.8]{Lind_Marcus_1995} it follows that for any sequence $(d_i)\in F_{
X_{p_L}}(\nu^m)\subseteq F_{\vs_{p_L}}(\nu^m)$ we have $x_1\ldots x_N\om d_1d_2\ldots\in F_{\vs_{p_L}}(x_1\ldots x_N)$. In other words,
\[
\set{x_1\ldots x_N \om d_1d_2\ldots: (d_i)\in F_{X_{p_L}}(\nu^m)}\subseteq F_{\vs_{p_L}}(x_1\ldots x_N).
\]
Therefore, by (\ref{e31}) it follows that
\begin{equation}\label{eq:kon1}
\begin{split}
  \dim_H(\v_{p_L}\cap(x-\de, x+\de))   &\ge \dim_H \pi_{p_L}(F_{\vs_{p_L}}(x_1\ldots x_N))
  \\&\ge\dim_H  \pi_{p_L}(F_{X_{p_L}}(\nu^m)) = \dim_H \pi_{p_L}(X_{p_L}),
  \end{split}
\end{equation}
where the last equality holds by the transitivity of $(X_{p_L}, \si)$.
Observe that $\pi_{p_L}(X_{p_L})$ is a graph-directed set satisfying the open set condition (cf.~\cite{Mauldin_Williams_1988}). Then the Hausdorff dimension of $\pi_{p_L}(X_{p_L})$ is given by 
\begin{equation}\label{eq:kon2}
\dim_H\pi_{p_L}(X_{p_L})=\frac{h_{top}(X_{p_L})}{\log p_L}.
\end{equation}  By  (\ref{e32}), (\ref{eq:kon1}), (\ref{eq:kon2}) and   Lemma \ref{l24} (i) we conclude  that
\begin{align*}
\dim_H(\v_{p_L}\cap(x-\de, x+\de)) &\ge\dim_H\pi_{p_L}(X_{p_L}) \\
&=\frac{h_{top}(X_{p_L})}{\log p_L}=\frac{h_{top}(\vs_{p_L})}{\log p_L}\\
&=\dim_H\u_{p_L}
 =\dim_H\v_{p_L}.\qedhere
\end{align*}
\end{proof}

Now we consider $q\in \B$. We need the following lemma.
\begin{lemma}\label{l34}
  Let $q\in(q_{KL}, M+1]$ and $x_1\ldots x_N\in\L(\vs_q)$. Let   $\set{p_n}\subset(1,M+1]$ be a sequence  such that     $\al(p_n)\in\vs$ for each $n\ge 1$,  and $p_n\nearrow q$ as $n\ra\f$. Then 
  \[x_1\ldots x_N\in \L(\vs_{p_n})\qquad\textrm{for all sufficiently large }n.\]
\end{lemma}
\begin{proof}
  Since $x_1\ldots x_N\in \L(\vs_q)$, we have
  \[
  \overline{\al_1(q)\ldots \al_{N-i}(q)}\lle x_{i+1}\ldots x_N\lle \al_1(q)\ldots \al_{N-i}(q)\qquad\textrm{for any}\quad 0\le i<N.
  \]
  Let $s\in\set{0,1,\ldots, N-1}$ be the smallest  integer such that
  \begin{equation}\label{e33}
    x_{s+1}\ldots x_N=\overline{\al_1(q)\ldots \al_{N-s}(q)}\quad\textrm{or}\quad x_{s+1}\ldots x_N=\al_1(q)\ldots \al_{N-s}(q).
  \end{equation}
  If there is no $s\in\set{0,1,\ldots, N-1}$ for which \eqref{e33} holds, then we set $s=N$. By our choice of $s$ it follows that
  \begin{equation}\label{e34}
    \overline{\al_1(q)\ldots \al_{N-i}(q)}\prec x_{i+1}\ldots x_N\prec \al_1(q)\ldots \al_{N-i}(q)\qquad\textrm{for all}\quad 0\le i<s.
  \end{equation}

In terms of \eqref{e33} we may assume by symmetry that
  \begin{equation}\label{e35}
  x_{s+1}\ldots x_N=\al_1(q)\ldots \al_{N-s}(q).
  \end{equation}
Since $p_n \nearrow q$ as $n\ra\f$, by Lemma \ref{l21} there exists  $K\in\N$ such that
  \[
  \al_1(p_n)\ldots\al_N(p_n)=\al_1(q)\ldots \al_N(q)\qquad\textrm{for any}\quad n\ge K.
  \]
By the assumption that $\al(p_n)\in\vs$  for any $n\ge 1$,   it follows from \eqref{e34} and \eqref{e35} that
  \[
  x_1\ldots x_N\al_{N-s+1}(p_n)\al_{N-s+2}(p_n)\ldots =x_1\ldots x_s\al_1(p_n)\al_2(p_n)\ldots\;\in\vs_{p_n}
  \]
  for any $n\ge K$. So, $x_1\ldots x_N\in \L(\vs_{p_n})$ for all $n\ge K$.
\end{proof}

\begin{lemma}\label{l35}
Let $q\in\B$.   Then for any $x\in\v_{q}$ we have
\[
\dim_H(\v_{q}\cap(x-\de, x+\de))=\dim_H\v_{q}\qquad\textrm{for any }~\de>0.
\]
\end{lemma}
\begin{proof}
Take $q\in\B$. Since $\B\subset(q_{KL}, M+1]\setminus\bigcup(p_L, p_R]$, by Lemma \ref{l27} (i) there exists a sequence of plateaus $\set{[p_L(n), p_R(n)]}_{n=1}^\f$ such that $p_L(n)\nearrow q$ as $n\ra\f$.

Now we fix $\de>0$ and  $x\in\v_q$.  Suppose $(x_i)\in\vs_q$ is a $q$-expansion of $x$. Then there exists a large integer $N$ such that
\begin{equation}\label{e37}
\pi_q(F_{\vs_q}(x_1\ldots x_N))\subseteq \v_q\cap(x-\de, x+\de).
\end{equation}
By Lemmas \ref{l28} (i) and   \ref{l29} (i) we  have  $\al(p_L(n))\in\vs$ for all $n\ge 1$. Then applying Lemma \ref{l34} to the sequence $\set{p_L(n)}$ gives a large integer $K$ such that
\[x_1\ldots x_N\in\L(\vs_{p_L(n)})\qquad\textrm{ for all }~ n\ge K.\]
Since $\vs_{p_L(n)}\subset\vs_{q}$ for any $n\ge 1$, it follows from  \eqref{e37}   that
\begin{equation}\label{e38}
  \pi_q(F_{\vs_{p_L(n)}}(x_1\ldots x_N))\subset \v_q\cap(x-\de, x+\de)\qquad\textrm{ for all }~ n\ge K.
\end{equation}
By  (\ref{e38}) and the proof of Lemma \ref{l33} it follows that for any $n\ge K$,
\begin{align*}
\dim_H(\v_q\cap(x-\de, x+\de))&\ge\dim_H\pi_q(F_{\vs_{p_L(n)}}(x_1\ldots x_N))\ge\frac{h_{top}(\vs_{p_L(n)})}{\log q}.
\end{align*}
Letting $n\ra\f$ we have $p_L(n)\nearrow q$, and then we conclude by the continuity of the function $q \mapsto h_{top}(\vs_q)$ (see Lemma \ref{l24} (ii)) that
\[
\dim_H(\v_q\cap(x-\de, x+\de))\ge  \frac{h_{top}(\vs_q)}{\log q}=\dim_H\u_q=\dim_H\v_q.  \qedhere
\]
\end{proof}

\begin{proof}[Proof of Theorem \ref{t31}]
Take $q\in(1, q_{KL}]\cup((q_{KL},M+1]\setminus\bigcup(p_L, p_R])$. If $q\in(1, q_{KL}]$, then the result follows  from the fact that $\dim_H\u_q=0$ (see Lemma \ref{l24}). 

Assume $q\in(q_{KL}, M+1]\setminus\bigcup(p_L, p_R]$ where the union is taken over all plateaus $[p_L, p_R]$ of $H$. Take $x\in\u_q$. If $x\notin\set{0, M/(q-1)}$, then by Lemma \ref{l32} $x$ belongs to an affine  copy of $\v_q$. Since the Hausdorff dimension is invariant under affine transformations (cf.~\cite{Falconer_1990}), the statement follows from Lemmas \ref{l33} and \ref{l35}.

So, it remains to consider $x=0$ and $x=M/(q-1)$. By symmetry we may assume  $x=0$. Take $\de>0$. Then by Lemma \ref{l32} there exists a sufficiently large integer $m$ such that
\[
\frac{1}{q^m}+\frac{\v_q}{q^m}\subseteq (\u_q\cup\pazocal{N})\cap(-\de, \de),
\]
where $\pazocal{N}$ is at most  countable. This proves the statement  for $x=0$. 
\end{proof}

{At the end of this section we strengthen Theorem \ref{t31} and give a complete characterization of the set 
\[\set{q\in(1,M+1]: \u_q\textrm{ is dimensionally homogeneous}}.\]
Let $[p_L, p_R]\subset(q_{KL}, M+1]$ be a plateau of $H$. Note that  $p_L\in\overline{\B}\setminus\B\subset\overline{\ub}\setminus\ub$. Then by \cite[Theorem 1.7]{DeVries_Komornik_2008} there exists a largest $\hat p_L\in(p_L, p_R)$ such that the set-valued map $q\mapsto \vs_q$ is constant in $[p_L, \hat p_L)$. Furthermore, for $q=\hat p_L$ any sequence in the difference set $\vs_{\hat p_L}\setminus\vs_{p_L}$ is not contained in $\us_{\hat p_L}$. Then 
 by the same argument as in the proof of Lemma \ref{l33} it follows that Theorem \ref{t31} also holds for any $q\in[p_L, \hat p_L]$. Clearly, $\u_q$ is dimensionally homogeneous for $q\le q_{KL}$. So, the univoque set $\u_q$ is dimensionally homogeneous for   any $q\in(1, q_{KL}]\cup((q_{KL}, M+1]\setminus\bigcup(\hat p_L, p_R])$. This, combined with some recent progress obtained by Allaart et al.~\cite{Allaart-Baker-Kong-17}, implies the following.
 \begin{theorem}\label{th:dimensionally homogeneous}\mbox{}
 
 \begin{enumerate}[{\rm(i)}]
 \item If $M=1$ or $M$ is even, then $\u_q$ is dimensionally homogeneous if, and only if, $ q\in(1, q_{KL}]\cup((q_{KL}, M+1]\setminus\bigcup(\hat p_L, p_R])$.
 
 \item If $M=2k+1\ge 3$, then $\u_q$ is dimensionally homogeneous if, and only if, $ q\in(1, q_{KL}]\cup((q_{KL}, M+1]\setminus\bigcup(\hat p_L, p_R])$ or $q=\frac{k+3+\sqrt{k^2+6k+1}}{2}$. 
 \end{enumerate}
 \end{theorem}
\begin{proof}
By Theorem \ref{t31} and the above arguments it follows that $\u_q$ is dimensionally homogeneous for any $q\in(1, q_{KL}]\cup((q_{KL}, M+1]\setminus\bigcup(\hat p_L, p_R])$. Then to prove the sufficiency it remains to prove the dimensional homogeneity of $\u_q$ for $q=\frac{k+3+\sqrt{k^2+6k+1}}{2}=:q_\star$ with $M=2k+1\ge 3$. Note that $q_\star$ is the right endpoint of the entropy plateau generated by $k+1$, i.e., $[p_\star, q_\star]$ is an entropy plateau with $\al(p_\star)=(k+1)^\f$ and $\al(q_\star)=(k+2)k^\f$. Then by \cite[Corollary 3.10]{Allaart-Baker-Kong-17} it follows that 
\begin{equation}\label{eq:dh-1}
h_{top}(\vs_{q_\star}\setminus\vs_{p_\star})=h_{top}(\vs_{p_\star})=\log 2,
\end{equation} 
where the second equality follows from that $\vs_{p_\star}=\set{k, k+1}^\N$.
Furthermore, any sequence in $\vs_{q_\star}\setminus\vs_{p_\star}$  eventually ends in a transitive sub-shift of finite type $(X, \si)$ with states $\set{k-1, k, k+1, k+2}$ and adjacency matrix 
\begin{equation}\label{eq:matrix}
A=\left(
\begin{array}{cccc}
0&0&1&1\\
1&1&0&0\\
0&0&1&1\\
1&1&0&0
\end{array}
\right).
\end{equation}
Observe that $h_{top}(X)=\log 2$. 
Using (\ref{eq:dh-1}) and by a similar argument as in the proof of Lemma \ref{l33} it follows that $\u_{q_\star}$ is dimensionally homogeneous. 

Now we prove the necessity. Without loss of generality we assume that $M=1$ or $M$ is even. Let $[p_L, p_R]\subset(q_{KL}, M+1]$ be an entropy plateau generated by $a_1\ldots a_m$, and let $\hat p_L\in(p_L, p_R)$ be the largest point such that the map $q\mapsto \vs_q$ is constant in $[p_L, \hat p_L)$. In fact, we have $\al(\hat p_L)=(a_1\ldots a_m^+\overline{a_1\ldots a_m^+})^\f$ (cf.~\cite{DeVries_Komornik_2008}).  Take $q\in(\hat p_L, p_R]$. Then $\mathbf W_q\setminus\vs_{p_L}\ne\emptyset$, where $\mathbf W_q$ is the set   of sequences $(x_i)$ satisfying  
  \[
  \overline{\al(q)}\prec\si^n((x_i))\prec \al(q)\quad\textrm{for any}\quad n\ge 0.
  \]
Furthermore, any sequence in $\mathbf W_q\setminus\vs_{p_L}$ must end  in the sub-shift of finite type $(Y, \si)$ with  states 
$
\set{\overline{a_1\ldots a_m^+}, ~\overline{a_1\ldots a_m}, ~a_1\ldots a_m, ~a_1\ldots a_m^+}
$
and adjacency matrix $A$ defined   in (\ref{eq:matrix}).  In particular, 
\begin{equation}\label{eq:hd-3}
h_{top}(Y)=\frac{\log 2}{m}=h_{top}(\vs_{p_R}\setminus\vs_{p_L})<h_{top}(\vs_{p_L}),
\end{equation}  
where the inequality follows from \cite[Corollary 3.10]{Allaart-Baker-Kong-17}. Observe that $\mathbf W_q\subseteq\us_q$. 
Therefore,   by  (\ref{eq:hd-3}) and the same argument as in the proof of Lemma  \ref{l33} it follows that for any $x\in\pi_q(\mathbf W_q\setminus\vs_{p_L})\subset\u_q$ there exists   $\de>0$ such that 
\[
\dim_H(\u_q\cap(x-\de, x+\de))\le\frac{h_{top}(Y)}{\log q}<\frac{h_{top}(\vs_{p_L})}{\log q}=\dim_H\u_q.
\]
This completes the proof.
\end{proof}
}

\section{Auxiliary Proposition}\label{sec:aux prop}
In this section we   prove an auxiliary proposition that will be used to prove Theorem \ref{t12} in the next section.

\begin{proposition}\label{p41}
Let $q\in\OB\setminus\set{M+1}$. Then for any $\ep>0$ there exists  $\de>0$ such that
\[
(1-\ep)\dim_H\pi_q(\BS_\de(q))\le \dim_H(\OB\cap(q-\de, q+\de))\le (1+\ep)\dim_H\pi_{q+\de}(\BS_\de(q)),
\]
where
\[
\BS_\de(q):=\set{\al(p): p\in\OB\cap(q-\de, q+\de)}.
\]
\end{proposition}

The proof of Proposition \ref{p41} is based on the following lemma for the Hausdorff dimension under H\"{o}lder continuous maps (cf.~\cite{Falconer_1990}).

\begin{lemma}\label{l42}
Let $f: (X, \rho_1)\ra(Y, \rho_2)$ be a H\"{o}lder map between two metric spaces, i.e., there exist two constants $C>0$ and $\la>0$ such that
\[
\rho_2(f(x), f(y))\le C\rho_1(x, y)^\la
\]
for any $x, y\in X$ with $\rho_1(x, y)\le c$ (here $c$ is a small constant). Then $\dim_H f(X)\le \frac{1}{\la}\dim_H X$.
\end{lemma}

First we prove the second inequality in Proposition \ref{p41}.
\begin{lemma}\label{l43}
Let $q\in\OB\setminus\set{M+1}$. Then for any $\ep>0$ there exists  $\de>0$ such that
\[
\dim_H(\OB\cap(q-\de, q+\de))\le (1+\ep) \dim_H \pi_{q+\de}(\BS_\de(q)).
\]
\end{lemma}
\begin{proof}
Fix $\ep>0$ and $q\in\OB\setminus\set{M+1}$. Then there exists  $\de>0$ such that
\begin{equation}\label{e41}
q-\de>1,\quad q+\de<M+1\quad\textrm{and}\quad \frac{\log (q+\de)}{\log(q-\de)}\le 1+\ep.
\end{equation}
Since $\OB\subseteq\overline{\ub}$, by Lemmas \ref{l21} and \ref{l23} (ii) it follows that for each $p\in\OB\cap(q-\de, q+\de)$  we have 
\[
\overline{\al(q+\de)}\prec\overline{\al(p)}\prec\si^i(\al(p))\lle \al(p)\prec \al(q+\de)\qquad\textrm{for all}\quad i\ge 0.
\]
So, by Lemma \ref{l22}   $\al(p)\in\us_{q+\de}$ for any $p\in\OB\cap(q-\de, q+\de)$. This implies that the map
\[ g\,: \, \OB\cap(q-\de, q+\de) \ra \pi_{q+\de}(\BS_\de(q)); \qquad p\mapsto \pi_{q+\de}(\al(p))
\]
is bijective. By Lemma \ref{l42} it suffices to prove that there exists a constant $C>0$ such that
\[
\big|\pi_{q+\de}(\al(p_2))-\pi_{q+\de}(\al(p_1))\big|\ge  C\,|p_2-p_1|^{1+\ep}
\]
 for any $p_1, p_2\in\OB\cap(q-\de, q+\de).$

Take $p_1, p_2\in\OB\cap(q-\de, q+\de)$ with $p_1<p_2$. Then by Lemma \ref{l21} we have
$\al(p_1)\prec\al(p_2)$. So,  there exists $n\ge 1$ such that
\[
\al_1(p_1)\ldots \al_{n-1}(p_1)=\al_1(p_2)\ldots\al_{n-1}(p_2)\quad\textrm{and}\quad \al_n(p_1)<\al_n(p_2).
\]
 Then
\begin{equation}\label{e42}
\begin{split}
0<p_2-p_1&=\sum_{i=1}^\f\frac{\al_i(p_2)}{p_2^{i-1}}-\sum_{i=1}^\f\frac{\al_i(p_1)}{p_1^{i-1}}\\
&\le\sum_{i=1}^{n-1}\left(\frac{\al_i(p_2)}{p_2^{i-1}}-\frac{\al_i(p_1)}{p_1^{i-1}}\right)+\sum_{i=n}^\f\frac{\al_i(p_2)}{p_2^{i-1}}~\le p_2^{2-n},
\end{split}
\end{equation}
where the last inequality follows from the property of quasi-greedy expansion $\al(p_2)$ that $\sum_{i=1}^\f\al_{k+i}(p_2)/p_2^i\le 1$ for any $k\ge 1$.

On the other hand, by (\ref{e41}) we have $\al(p_2)\lle\al(q+\de)\prec\al(M+1)=M^\f$. Then there exists a large integer $N$ (depending on $q+\de$) such that
\begin{equation}\label{e43}
\al_1(p_2)\ldots\al_N(p_2)\lle  M^{N-1}(M-1). 
\end{equation}
Note that $p_2\in\OB\subseteq\overline{\ub}$. Then by Lemma \ref{l23} (ii) and (\ref{e43}) it follows that
\[\al_{m+1}(p_2)\al_{m+2}(p_2)\ldots \succ\overline{\al(p_2)}\lge 0^{N-1}10^\f\qquad\textrm{for any}\quad m\ge 1.\]
This  implies that
\begin{align*}
&\quad\pi_{q+\de}(\al(p_2))-\pi_{q+\de}(\al(p_1))\\
&=\sum_{i=1}^\f\frac{\al_i(p_2)-\al_i(p_1)}{(q+\de)^i}\\
&=\frac{\al_n(p_2)-\al_n(p_1)}{(q+\de)^n}-\frac{1}{(q+\de)^n}\sum_{i=1}^\f\frac{\al_{n+i}(p_1)}{(q+\de)^i}+\sum_{i=n+1}^\f\frac{\al_i(p_2)}{(q+\de)^i}\\
&\ge\frac{1}{(q+\de)^n}-\frac{1}{(q+\de)^n}\sum_{i=1}^\f\frac{\al_{n+i}(p_1)}{p_1^i}+\sum_{i=n+1}^\f\frac{\al_i(p_2)}{(q+\de)^i}\\
&\ge \sum_{i=n+1}^\f\frac{\al_i(p_2)}{(q+\de)^i}~\ge\frac{1}{(q+\de)^{n+N}},
\end{align*}
where the second inequality follows from the same property of the quasi-greedy expansion $\al(p_1)$ that was used before.

Therefore, by (\ref{e41}) and (\ref{e42}) it follows that
\begin{align*}
\pi_{q+\de}(\al(p_2))-\pi_{q+\de}(\al(p_1))&\ge \big((q+\de)^{-\frac{n+N}{1+\ep}}\big)^{1+\ep}\\
&\ge \big((q-\de)^{-n-N}\big)^{1+\ep}\\
&\ge (q-\de)^{-N(1+\ep)} ({p_2^{-n}})^{1+\ep}\ge C (p_2-p_1)^{1+\ep},
\end{align*}
where  the constant $C=(q-\de)^{-N(1+\ep)}(q+\de)^{-2(1+\ep)}$. This completes the proof.
\end{proof}

Now we turn to prove  the first inequality of Proposition \ref{p41}.
\begin{lemma}\label{l44}
Let $q\in\OB\setminus\set{M+1}$. Then for any $\ep>0$ there exists  $\de>0$ such that
\[
\dim_H(\OB\cap(q-\de, q+\de))\ge (1-\ep)\dim_H\pi_q(\BS_\de(q)).
\]
\end{lemma}
\begin{proof}
The proof is similar to that of Lemma \ref{l43}. Fix $\ep>0$ and take $q\in\OB\setminus\set{M+1}$. Then there exists  $\de>0$ such that
\begin{equation}\label{e44}
q-\de>1, \quad q+\de<M+1\quad\textrm{and}\quad \frac{\log (q+\de)}{\log q}\le \frac{1}{1-\ep}.
\end{equation}

Take $p_1, p_2\in\OB\cap(q-\de, q+\de)$ with $p_1<p_2$. Then by Lemma \ref{l21} we have $\al(p_1)\prec \al(p_2)$, and therefore there exists a smallest integer $n\ge 1$ such that $\al_n(p_1)<\al_n(p_2)$. This implies that
\begin{equation}\label{e45}
\big|\pi_q(\al(p_2))-\pi_q(\al(p_1))\big|=\left |\sum_{i=1}^\f\frac{\al_i(p_2)-\al_i(p_1)}{q^i}\right |\le\sum_{i=n}^\f\frac{M}{q^i}=\; \frac{M q}{q-1}q^{-n}.
\end{equation}

On the other hand, observe that $q+\de<M+1$. Then $\al(p_2)\lle\al(q+\de)\prec\al(M+1)=M^\f$. So, there exists  $N\ge 1$ such that
\[\al_1(p_2)\ldots \al_N(p_2)\lle M^{N-1}(M-1).\] 
Since $p_2\in\OB\subseteq \overline{\ub}$, Lemma \ref{l23} (ii) gives
\[
1=\sum_{i=1}^\f\frac{\al_i(p_2)}{p_2^i}>\sum_{i=1}^n\frac{\al_i(p_2)}{p_2^i}+\frac{1}{p_2^{n+N}},
\]
which implies that
\begin{equation}\label{e46}
\begin{split}
\frac{1}{p_2^{n+N}} &<1-\sum_{i=1}^n\frac{\al_i(p_2)}{p_2^i}=\sum_{i=1}^\f\frac{\al_i(p_1)}{p_1^i}-\sum_{i=1}^n\frac{\al_i(p_2)}{p_2^i}\\
&\le \sum_{i=1}^n\left(\frac{\al_i(p_2)}{p_1^i}-\frac{\al_i(p_2)}{p_2^i}\right) \\
& \le \sum_{i=1}^\f\left(\frac{M}{p_1^i}-\frac{M}{p_2^i}\right) =  \frac{M}{(p_1-1)(p_2-1)}(p_2-p_1).
\end{split}
\end{equation}
Here the second inequality holds since 
\begin{align*}
&\al_1(p_1)\ldots \al_{n-1}(p_1)=\al_1(p_2)\ldots \al_{n-1}(p_2),\\
&  \al_n(p_1)<\al_n(p_2)\quad\textrm{and}\quad  \sum_{i=1}^\f{\al_{n+i}(p_1)}/{p_1^i}\le 1.
\end{align*}

Therefore, by (\ref{e44})--(\ref{e46}) we conclude that
\begin{align*}
\big|\pi_q(\al(p_2))-\pi_q(\al(p_1))\big| &\le \frac{M q^{N+1}}{q-1}\big(q^{-\frac{n+N}{1-\ep}}\big)^{1-\ep}\\
&\le \frac{M q^{N+1}}{q-1}(q+\de)^{-(n+N)(1-\ep)}\\
&\le \frac{M q^{N+1}}{q-1}p_2^{-(n+N)(1-\ep)}\;<C (p_2-p_1)^{1-\ep},
\end{align*}
where
\[C=\frac{M^{2-\ep} q^{N+1}}{(q-1)(q-\de-1)^{2(1-\ep)}}.\]
Note by Lemma \ref{l21} that the map $p\mapsto \al(p)$ is bijective from $\overline{\B}\cap(q-\de, q+\de)$ onto $\BS_\de(q)$. Hence, the lemma  follows by letting $f=\pi_q \circ \alpha$ in Lemma \ref{l42}.
\end{proof}
\begin{proof}[Proof of Proposition \ref{p41}]
The proposition follows from Lemmas \ref{l43} and \ref{l44}.
\end{proof}

\section{Local dimension of $\B$}\label{sec:proof of thm1}
In this section we will prove   Theorem \ref{t12}, which states that  for 
 any $q\in\OB$ we have
  \[\lim_{\de\ra 0}\dim_H(\OB\cap(q-\de, q+\de))= \dim_H\u_q.\]
 First we prove the upper bound.
\begin{proposition}
\label{p52}
For any  $q\in\OB$ we have
\[
\lim_{\de\ra 0}\dim_H(\OB\cap(q-\de, q+\de))\le \dim_H\u_q.
\]
\end{proposition}
\begin{proof}
Take $q\in\overline{\B}$. By Lemma \ref{l24}  and Proposition \ref{p41}  it follows that for any $\ep>0$ there exists a $\de>0$ such that
\begin{equation}\label{e51}
\begin{split}
 \dim_H\u_{q+\de}&\le \dim_H\u_q+\ep,\\
\dim_H(\OB\cap(q-\de, q+\de))&\le (1+\ep)\dim_H\pi_{q+\de}(\BS_\de(q)),
\end{split}
\end{equation}
where $\BS_\de(q)=\set{\al(p): p\in(q-\de, q+\de)\cap\OB}$.

Since $\OB\subseteq\overline{\ub}$,   Lemmas \ref{l21} and \ref{l23} (ii) give that any sequence $\al(p)\in \BS_\de(q)$ satisfies
   \[
   \overline{\al(q+\de)}\prec\overline{\al(p)}\prec\si^n(\al(p))\lle\al(p)\prec\al(q+\de)\qquad\textrm{for all}\quad n\ge 0.
   \]
  By Lemma \ref{l22} this implies that $ \BS_\de(q) \subseteq\us_{q+\de}$. Therefore, by (\ref{e51})   it follows that
\begin{align*}
\dim_H(\OB\cap(q-\de, q+\de)) &\le (1+\ep)\dim_H\pi_{q+\de}(\BS_\de(q))\\
 &\le(1+\ep)\dim_H\u_{q+\de} \le(1+\ep)(\dim_H\u_q+\ep).
\end{align*}
Since $\ep>0$ was arbitrary, this completes the proof.
\end{proof}

The proof of the  lower bound of Theorem \ref{t12}  is tedious. We will prove this in several steps.   First we need  the following  lemma.

 \begin{lemma}
 \label{l53}
  Let $[p_L, p_R]\subseteq(q_{KL}, M+1)$ be a plateau of $H$ such that  $\al(p_L)=(\al_1\ldots \al_m)^\f$ with period $m$. Then
    \begin{align*}
         &\al_{i+1}\ldots\al_m\prec \al_1\ldots\al_{m-i}&&\textrm{for all}\quad 0< i<m,\\
         &\al_{i+1}\ldots \al_m\al_1\ldots \al_i\succ \overline{\al_1\ldots \al_m}&&\textrm{for all}\quad 0\le i<m.
  \end{align*}
 \end{lemma}
 \begin{proof}
Since $(\al_1\ldots \al_m)^\f$ is the quasi-greedy $p_L$-expansion of $1$ with  period $m$,
 the greedy $p_L$-expansion of $1$ is $\al_1\ldots \al_m^+ 0^\f$. So, by \cite[Propostion 2.2]{Vries-Komornik-Loreti-2016} it follows that
 $\si^n(\al_1\ldots\al_m^+0^\f)\prec \al_1\ldots\al_m^+ 0^\f$ for any $n\ge 1$. This implies
 \[
 \al_{i+1}\ldots \al_m\prec \al_{i+1}\ldots \al_m^+\lle \al_1\ldots \al_{m-i}\qquad\textrm{for any}\quad 0<i<m.
 \]

  Lemma \ref{l26} states  that $p_L\in\OB\subset\overline{\ub}$. Then by Lemma \ref{l23} (ii) we have that
 \begin{equation*}
(\al_{i+1}\ldots \al_m \al_1\ldots \al_i)^\f= \si^i((\al_1\ldots \al_m)^\f)\succ(\overline{\al_1\ldots \al_m})^\f 
 \end{equation*}
 {for any} $0\le i<m.$
 This implies that 
 \[\al_{i+1}\ldots \al_m\al_1\ldots \al_i\succ \overline{\al_1\ldots \al_m}\qquad \textrm{for any }0\le i<m.\qedhere\]
 \end{proof}

Let $[p_L, p_R]\subset(q_{KL}, M+1)$ be a plateau of $H$.  For any $N\ge 1$ let  $(\ws_{p_L,N}, \si)$ be a subshift of finite type in $\set{0,1,\ldots, M}^\N$ with the set of forbidden blocks  $c_1\ldots c_N$ satisfying 
\[ c_1\ldots c_N\lle \overline{\al_1(p_L)\ldots\al_N(p_L)}\qquad \textrm{or}\qquad c_1\ldots c_N\lge \al_1(p_L)\ldots \al_N(p_L).\]
Then any sequence $(x_i)\in\ws_{p_L,N}$ satisfies
\[
\overline{\al_1(p_L)\ldots\al_N(p_L)}\prec \si^n((x_i))\prec \al_1(p_L)\ldots\al_N(p_L)\qquad\textrm{for all}\quad n\ge 0.
\]
If $\al_N(p_L)>0$, then    $\ws_{p_L, N}$  is indeed the set of sequences $(x_i)\in\set{0,1,\ldots, M}^\N$ satisfying 
\[
(\overline{\al_1(p_L)\ldots \al_{N}(p_L)}^+)^\f\lle\si^n((x_i))\lle (\al_1(p_L)\ldots \al_N(p_L)^-)^\f
\]  
for all $n\ge 0$.
By the definition of $\ws_{p_L, N}$ it gives that 
\[
   \ws_{p_L,1}\subseteq\ws_{p_L,2}\subseteq\cdots\subseteq\vs_{p_L}.
\]
We emphasize  that $\ws_{p_L, 1}$ can be an empty set, and the inclusions in the above equation are not necessarily  strict. 

Observe that $(\vs_{p_L}, \si)$ is a subshift of finite type with positive topological entropy. The following asymptotic result  was established in
\cite[Proposition 2.8]{Komornik_Kong_Li_2015_1}.
 \begin{lemma}\label{l54}
Let $[p_L, p_R]\subseteq[q_{T}, M+1]$ be a plateau of $H$. Then
\[\lim_{N\ra\f}h_{top}(\ws_{p_L,N})= h_{top}(\vs_{p_L}).\]
\end{lemma}

Recall from \eqref{e28} that  
\begin{align*}
\xi(n)&=\la_1\ldots\la_{2^{n-1}}(\overline{\la_1\ldots\la_{2^{n-1}}}\,^+)^\f&\textrm{if}\quad M=2k,\\
\xi(n)&=\la_1\ldots\la_{2^n}(\overline{\la_1\ldots\la_{2^n}}\,^+)^\f&\textrm{if}\quad M=2k+1.
\end{align*}
Note  that the sequence $(\la_i)$ in the definition of $\xi(n)$ depends on $M$. 
In the following lemma we show that the  entropy of $(\ws_{p_L,N}, \si)$ is equal to the entropy of the follower set $F_{\ws_{p_L,N}}(\nu)$ for all sufficiently large integers $N$, where $\nu$ is the word defined in Lemma \ref{l28} (iii) or Lemma \ref{l29} (iii).

\begin{lemma}\label{l55}
\mbox{}
\begin{enumerate}[{\rm(i)}]
\item Let $[p_L, p_R]\subset[q_T, M+1]$ be a plateau of $H$, and let 
\[
\nu=\left\{
\begin{array}{lll}
k&\textrm{if}& M=2k,\\
(k+1)k&\textrm{if}& M=2k+1.
\end{array}\right.
\] 
Then for all sufficiently large integers $N$ we have
\[h_{top}(F_{\ws_{p_L,N}}(\nu^\ell))=h_{top}(\ws_{p_L, N})\qquad\textrm{for any}\quad \ell\ge 1.\]

\item  Let $[p_L, p_R]\subset(q_{KL}, q_T)$ be a plateau of $H$ with $\xi(n+1)\lle \al(p_L)\prec \xi(n)$. Set
\[
\nu=\left\{
\begin{array}{lll}
\la_1\ldots\la_{2^n}^-&\textrm{if}& M=2k,\\
\la_1\ldots\la_{2^{n+1}}^-&\textrm{if}& M=2k+1.
\end{array}\right.
\]
Then  for all sufficiently large integers $N$ we have
\[h_{top}(F_{\ws_{p_L,N}}(\nu^\ell))=h_{top}(\ws_{p_L, N})\qquad\textrm{for any}\quad \ell\ge 1.\]
\end{enumerate}
\end{lemma}

\begin{proof}
Take $\ell\ge 1$.
First we prove (i). By Lemma \ref{l28} (iii) there exists a large integer $N\ge 2$ such that
$\nu^\ell\in\L(\ws_{p_L, N})$. Since $(\ws_{p_L,N}, \si)$ is a subshift of finite type, to prove (i) it suffices to prove that for any word $\rho\in\L(\ws_{p_L,N})$ there exists a word $\ga$ of uniformly bounded length   for which $\nu^\ell\ga\rho\in\L(\ws_{p_L, N})$.

Take $\rho=\rho_1\ldots \rho_m\in\L(\ws_{p_L,N})$. If $M=2k$, then  $\nu=k$. Since  $\al(p_L)\lge\al(q_T)=(k+1)k^\f$, we have
\[
\overline{\al_1(p_L)}\le k-1<\nu<k+1\le \al_1(p_L).
\]
So, $\nu^\ell\ga\rho\in\L(\ws_{p_L, N})$ by taking $\ga=\epsilon$ the empty word.
Similarly, if $M=2k+1$ then $\nu=(k+1)k$. Observe that $\al(p_L)\lge \al(q_T)=(k+1)((k+1)k)^\f$. This implies that $\nu^\ell\ga\rho\in\L(\ws_{p_L, N})$ by taking $\ga=\epsilon$ if the initial  word $\rho_1\ge k+1$, and by taking $\ga=k+1$ if $\rho_1\le k$.

Now we turn to prove (ii).  We only give the proof for $M=2k$, since the proof for $M=2k+1$ is similar. Then $\nu=\la_1\ldots\la_{2^n}^-$. By Lemma \ref{l29} (iii) there exists a large integer $N\ge 2^{n+1}$ such that $\nu^\f=(\la_1\ldots\la_{2^n}^-)^\f\in \ws_{p_L,N}$.  Since $h_{top}(\vs_{p_L})>0$, by Lemma \ref{l54} we can choose   $N$ sufficiently large  such that $h_{top}(\ws_{p_L, N})>0$.  Since $\ws_{p_L, N}$ is a subshift of finite type,  there exists a transitive subshift of finite type $X_N\subset\ws_{p_L,N}$ for which $h_{top}(X_N)=h_{top}(\ws_{p_L, N})$ (cf.~\cite[Theorem 4.4.4]{Lind_Marcus_1995}). We  claim that the word $\la_1\ldots\la_{2^n}$ or   $\overline{\la_1\ldots\la_{2^n}}$ belongs to $\L(X_{N})$.

By \eqref{e28} and \eqref{eq:lambda} it follows that
\[\xi(n)=\la_1\ldots\la_{2^{n-1}}(\overline{\la_1\ldots\la_{2^{n-1}}}\,^+)^\f=\la_1\ldots\la_{2^n}(\overline{\la_1\ldots\la_{2^{n-1}}}\,^+)^\f.\]
  Then the assumption $\xi(n+1)\lle\al(p_L)\prec\xi(n)$ gives that
\begin{equation}\label{e52}
\al_1(p_L)\ldots\al_{2^n}(p_L)= \la_1\ldots\la_{2^{n}}=\al_1(q_{KL})\ldots\al_{2^n}(q_{KL}).
\end{equation}
 Suppose that the words $\la_1\ldots\la_{2^n}$ and    $\overline{\la_1\ldots\la_{2^n}}$ do  not belong to $\L(X_{N})$.  Then by \eqref{e52} we have
\[X_N\subset\ws_{p_L, 2^n}=\ws_{q_{KL}, 2^n}\subset\vs_{q_{KL}}.\]
So, by Lemma \ref{l24} it follows that  $X_N$ has zero topological entropy, leading   to a contradiction with $h_{top}(X_N)=h_{top}(\ws_{p_L,N})>0$.

By the claim, to finish the proof of (ii) it suffices to prove that  for any word $\rho\in\L(X_N)$ with a prefix $\la_1\ldots\la_{2^n}$ or   $\overline{\la_1\ldots\la_{2^n}}$ there exists a word $\ga$ of uniformly bounded length  such that $\nu^\ell\ga\rho\in\L(\ws_{p_L, N})$. In \cite[Lemma 4.2]{Kong_Li_2015} (see also, \cite[Lemma 4.2]{AlcarazBarrera-Baker-Kong-2016}) it was shown that for any $n\ge 1$ we have
\begin{equation*}
\overline{\la_1\ldots\la_{2^n-i}}\prec\la_{i+1}\ldots\la_{2^n}\lle\la_1\ldots\la_{2^n-i}\qquad\textrm{for any}\quad 0\le i<2^n.
\end{equation*}
This implies that for any $0\le i<2^n$ we have
\begin{equation}\label{e53}
  \la_{i+1}\ldots\la_{2^n}^-\prec\la_1\ldots\la_{2^n-i}\quad\textrm{and}\quad \la_{i+1}\ldots\la_{2^n}^-\la_1\ldots\la_i\succ\overline{\la_1\ldots\la_{2^n}}.
\end{equation}
Observe that 
\[\nu=\la_1\ldots\la_{2^n}^-=\la_1\ldots\la_{2^{n-1}}\overline{\la_1\ldots\la_{2^{n-1}}}.\]
 Then by \eqref{e52} and \eqref{e53} it follows that if $\la_1\ldots\la_{2^n}$ is a prefix of $\rho$, then  $\nu^\ell\ga\rho\in\L(\ws_{p_L, N})$ by taking  $\ga=\epsilon$ the empty word, and if $\overline{\la_1\ldots\la_{2^n}}$ is a prefix of $\rho$ then $\nu^\ell\ga\rho\in\L(\ws_{p_L, N})$ by taking  $\ga=\la_1\ldots\la_{2^{n-1}}$.
\end{proof}

In the following lemma   we prove the lower bound of Theorem \ref{t12}  for $q\in[q_T, M+1]$ being the left endpoint of  an entropy  plateau.

 \begin{lemma}\label{l56}
 Let $[p_L, p_R]\subseteq[q_T, M+1]$ be a plateau  of $H$. Then for any $\de>0$ we have
 \[
 \dim_H(\OB\cap(p_L-\de, p_L+\de))\ge \dim_H\u_{p_L}.
 \]
  \end{lemma}
  \begin{proof}
 By Lemma \ref{l28} (i)  it follows that
$\al(p_L)=(\al_i)=(\al_1\ldots \al_m)^\f$
  is an irreducible  sequence, where   $m$ is the minimal period of  $\al(p_L)$.  Then,   there exists a large integer $N_1>m$ such that
 \begin{equation}\label{e54}
 \al_1\cdots\al_j(\overline{\al_1\ldots\al_j}\,^+)^\f\prec \al_1\ldots \al_{N_1}\qquad\textrm{if}\quad (\al_1\ldots\al_j^-)^\f\in\vs\textrm{  and  }1\le j\le m.
\end{equation}

Let $\nu$ be the word defined in     Lemma \ref{l55} (i).   Then by Lemma \ref{l28} (iii) there exist a large integer $N>N_1$ and a word   $\om$ such that
   \begin{equation}\label{e55}
\overline{\al_1\ldots \al_{N}}\prec\si^n(\al_1\ldots \al_m \om \nu^\f)\prec \al_1\ldots \al_{N}\quad\textrm{for any}\quad  n\ge 0.
\end{equation}
 Observe that $(\ws_{p_L, N}, \si)$ is an $N$-step subshift of finite type. Note by \eqref{e55} that $\al_1\ldots \al_m\om \nu^N\in\L(\ws_{p_L,N})$. Then by \cite[Theorem 2.1.8]{Lind_Marcus_1995} it follows that for any sequence $(d_i)\in F_{\ws_{p_L, N}}(\nu^N)$ we have $\al_1\ldots \al_m\om d_1d_2\ldots\in F_{\ws_{p_L, N}}(\al_1\ldots\al_m)$. In other words, 
\[
\set{\al_1\ldots\al_m\om d_1d_2\ldots: (d_i)\in F_{\ws_{p_L, N}}(\nu^N)}\subseteq F_{\ws_{p_L, N}}(\al_1\ldots\al_m)\subseteq \ws_{p_L,N}.
\] 
So,
\[
h_{top}(F_{\ws_{p_L, N}}(\nu^N))\le h_{top}(F_{\ws_{p_L, N}}(\al_1\ldots \al_m))\le h_{top}(\ws_{p_L,N}).
\]
Therefore, by Lemma \ref{l55} (i) we obtain
\begin{equation}
  \label{e56}
  h_{top}(F_{\ws_{p_L,N}}(\al_1\ldots\al_m))=h_{top}(\ws_{p_L,N}).
\end{equation}

Let $\La_N$ be the set of sequences $(a_i)\in\set{0,1,\ldots,M}^\f$ satisfying
\[
a_1\ldots a_{mN}=(\al_1\ldots \al_m)^{N} \quad\textrm{and}\quad a_{mN+ 1}a_{mN+ 2}\ldots \in F_{\ws_{p_L,N}}(\al_1\ldots\al_m).
\]
Fix $\de>0$. We claim that
\[\La_N\subseteq \BS_\de(p_L)=\set{\al(q): q\in\OB\cap(p_L-\de, p_L+\de)}\]
 for all sufficiently large integers $N>N_1$.

 Clearly, when $N$ increases  the  length of the common prefix  of sequences in $\La_N$ grows, and it coincides with a prefix of $\al(p_L)=(\al_1\ldots \al_m)^\f$. So, by Lemmas \ref{l21} and \ref{l220} it suffices to show that for all $N>N_1$ any sequence $(a_i)\in\La_N$  is  irreducible.

Take $N>N_1$   and     $(a_i)\in\La_N$.  First we claim  that
\begin{equation}\label{e57}
\overline{\al_1\ldots \al_{ N}}\prec \si^n((a_i))\prec \al_1\ldots \al_{ N} \qquad \textrm{for any }n\ge 1.
\end{equation}
Observe that
$
a_1\ldots a_{mN} =(\al_1\ldots\al_m)^N$ {and} the tails $a_{mN+ 1}a_{mN +2}\ldots \in F_{\ws_{p_L,N}}(\al_1\ldots\al_m).
$
Since $N>N_1>m$,  (\ref{e57})   follows directly from Lemma \ref{l53}.

Note by   the definition of $\Lambda_N$ that $a_1\ldots a_N=\al_1\ldots \al_N$. By \eqref{e57}  it follows that
$(a_i)\in\vs$.
So,  by Definition \ref{def:irreducible}  it remains to prove that
 \begin{equation}\label{e58}
 a_1\ldots a_j(\overline{a_1\ldots a_j}\,^+)^\f\prec (a_i)\qquad\textrm{whenever}\quad (a_1\ldots a_j^-)^\f\in\vs.
 \end{equation}
  We split the proof of (\ref{e58}) into the following three cases.
 \begin{itemize}

 \item For $1\le j\le m$,  (\ref{e58}) follows from (\ref{e54}).

 \item  For $m<  j\le N$,
 let $j=j_1  m +r_1$ with $j_1\ge 1$ and  $r_1\in\set{1,2,\ldots, m}$.
  Since $(\al_1\ldots\al_j^-)^\f=((\al_1\ldots\al_m)^{j_1}\al_1\ldots \al_{r_1}^-)^\f\in\vs$, we have
\[
\al_{r_1+1}\ldots \al_m\al_1\ldots\al_{r_1}\succ\al_{r_1+1}\ldots \al_m\al_1\ldots\al_{r_1}^-\lge\overline{\al_1\ldots\al_m}.
\]
This   implies that
\begin{align*}
a_1\ldots a_j(\overline{a_1\ldots a_j}\,^+)^\f&=(\al_1\ldots \al_m)^{j_1}\al_1\ldots \al_{r_1}\overline{\al_1\ldots \al_m}\ldots\\
&\prec (\al_1\ldots \al_m)^{j_1}\al_1\ldots\al_{r_1}\al_{r_1+1}\ldots \al_m\al_1\ldots \al_{r_1}0^\f\\
&\lle\;(a_i).
\end{align*}

  \item For $j> N$,   by (\ref{e57}) it follows that
 \[
(\overline{a_1\ldots a_j}\,^+)^\f=(\overline{\al_1\ldots\al_N a_{N+1}\ldots a_j}^+)^\f\prec a_{j+1}a_{j+2}\ldots,
 \]
which  implies that  (\ref{e58}) also holds in this case.
 \end{itemize}
 Therefore, $(a_i)$ is an irreducible  sequence, and thus $(a_i)\in \BS_\de(p_L)$. So,  $\La_N\subseteq \BS_{\de}(p_L)$ for all  $N>N_1$.

 Note that $\pi_{p_L}(\La_N)$ is a scaling copy  of $\pi_{p_L}(F_{\ws_{p_L,N}}(\al_1\ldots\al_m))$ which is related to  a graph-directed set satisfying the open set condition (cf.~\cite[Lemma 3.2]{Komornik_Kong_Li_2015_1}).  By Proposition \ref{p41} and  (\ref{e56}) it follows that for any $\ep>0$ there exists  $\de>0$ such that
 \begin{align*}
 \dim_H(\OB\cap(p_L-\de, p_L+\de))&\ge (1-\ep)\dim_H\pi_{p_L}(\BS_\de(p_L))\\
 &\ge (1-\ep)\dim_H\pi_{p_L}(\La_N)\\
 &=(1-\ep)\frac{h_{top}(F_{\ws_{p_L,N}}(\al_1\ldots\al_m))}{\log p_L}\\
 &=(1-\ep)\frac{h_{top}(\ws_{p_L,N})}{\log p_L}
 \end{align*}
 for all sufficiently large integers  $N>N_1$.  Letting $N\ra\f$  we conclude by Lemmas \ref{l54}  and  \ref{l24}   that
 \[
 \dim_H(\OB\cap(p_L-\de, p_L+\de))\ge (1-\ep)\frac{h_{top}(\vs_{p_L})}{\log p_L}=(1-\ep)\dim_H\u_{p_L}.
 \]
Since $\ep>0$ was taken arbitrarily,  this establishes the lemma.
\end{proof}

Now we prove the lower bound of Theorem \ref{t12}  for $q\in(q_{KL}, q_T)$ being the left endpoint of an entropy plateau.
   \begin{lemma}
  \label{l57} 
  Let $[p_L, p_R]\subset(q_{KL}, q_T)$ be a plateau of $H$. Then for any $\de>0$ we have
  \[
  \dim_H\OB\cap(p_L-\de, p_L+\de)\ge \dim_H\u_{p_L}.
  \]
\end{lemma}
\begin{proof}
  The proof is similar to that of Lemma \ref{l56}. We only give the proof for $M=2k$, since the proof for $M=2k+1$ is similar.

  By Lemma \ref{l29} (i) it follows that $\al(p_L)=(\al_i)=(\al_1\ldots \al_m)^\f$ is a $*$-irreducible sequence, where $m$ is the  minimal period of $\al(p_L)$. Then  there exists  $n\ge 1$ such that
  $
 \xi(n+1)\lle\al(p_L)\prec \xi(n),
  $
  where $\xi(n)=\la_1\ldots\la_{2^{n-1}}(\overline{\la_1\ldots\la_{2^{n-1}}}\,^+)^\f$. By \eqref{eq:lambda} this implies that $m>2^n$. Since $\al(p_L)=(\al_i)$ is  periodic while $\xi(n+1)$ is eventually periodic, we have $\xi(n+1)\prec \al(p_L)\prec \xi(n)$. So there exists a large integer $N_0$ such that
  \begin{equation}\label{eq:k1}
  \xi(n+1)\prec\al_1\ldots\al_{N_0}\prec\xi(n).
  \end{equation}
Since $\al(p_L)=(\al_i)$ is $*$-irreducible,   by Definition \ref{def:*-irreducible} there exists an integer $N_1>N_0$ such that
  \begin{equation}\label{e59}
    \al_1\ldots\al_j(\overline{\al_1\ldots\al_j}\,^+)^\f\prec \al_1\ldots\al_{N_1}\qquad\textrm{if}\quad (\al_1\ldots\al_j^-)^\f\in\vs\textrm{  and }2^n<j\le m.
  \end{equation}

 Let $\nu=\la_1\ldots\la_{2^n}^-$ be the word defined as in   Lemma \ref{l55} (ii). Then  by Lemma \ref{l29} (iii) there exist a  large integer $N \ge N_1$ and a word $\om$ such that
  \begin{equation}\label{e510}
    \overline{\al_1\ldots\al_{N }}\prec \si^j(\al_1\ldots\al_m\om \nu^\f)\prec \al_1\ldots\al_{N }\quad\textrm{for any}\quad j\ge 0.
  \end{equation}
 Observe that $(\ws_{p_L,N}, \si)$ is an $N$-step subshift of finite type. Note by \eqref{e510} that $\al_1\ldots\al_m\om\nu^N\in \L(\ws_{p_L, N})$. Then by  \cite[Theorem 2.1.8]{Lind_Marcus_1995} it follows that for any sequence $(d_i)\in F_{\ws_{p_L, N}}(\nu^N)$ we have $\al_1\ldots \al_m\om d_1d_2\ldots\in F_{\ws_{p_L, N}}(\al_1\ldots\al_m)$. This implies 
\[
\set{\al_1\ldots\al_m\om d_1d_2\ldots: (d_i)\in F_{\ws_{p_L, N}}(\nu^N)}\subseteq F_{\ws_{p_L, N}}(\al_1\ldots \al_m)\subseteq \ws_{p_L,N}.
\]
So, by
Lemma \ref{l55} (ii) we obtain
  \begin{equation}\label{e511}
    h_{top}(F_{\ws_{p_L,N}}(\al_1\ldots\al_m))=h_{top}(\ws_{p_L,N}).
  \end{equation}

  Let $\Delta_N$ be the set of sequences $(a_i)$ satisfying
  \[
  a_1\ldots a_{mN }=(\al_1\ldots\al_m)^N \quad\textrm{and}\quad a_{mN +1}a_{mN +2}\ldots\in F_{\ws_{p_L,N}}(\al_1\ldots\al_m).
  \]
  Fix $\de>0$. Then we claim that
    \[
\Delta_N\subset \BS_\de(p_L)=\set{\al(q): q\in\OB\cap(p_L-\de, p_L+\de)}
  \]
  for all sufficiently large integers $N>N_1.$
  Observe that the common prefix of sequences in $\Delta_N$ has length at least $m(N+1)$ and it coincides with a prefix of $\al(p_L)=(\al_1\ldots\al_m)^\f$. So, by Lemmas \ref{l21} and \ref{l220} it suffices to show that for all integers    $N>N_1$ any sequence in $\Delta_N$ is $*$-irreducible.

  Take $N>N_1$ sufficiently large and take $(a_i)\in\Delta_N$. Then by \eqref{eq:k1} we have $\xi(n+1)\prec(a_i)\prec\xi(n)$. Furthermore, by Lemma \ref{l53}  and the definition of $\Delta_N$ it follows that
  \begin{equation}
  \label{eq:512}
  \overline{a_1\ldots a_N}\prec \si^j((a_i))\prec a_1\ldots a_N\quad\textrm{for any}\quad j\ge 1.
  \end{equation}
  This implies that $(a_i)\in\vs$. Furthermore, by \eqref{e59}, \eqref{eq:512} and   arguments similar to those in
   the proof of Lemma \ref{l56} we can prove that
\[
a_1\ldots a_j(\overline{a_1\ldots a_j}\,^+)^\f\prec (a_i)
\]
whenever $j>2^n$ and $(a_1\ldots a_j^-)^\f\in\vs$. Therefore, by Definition \ref{def:*-irreducible} the sequence $(a_i)$ is $*$-irreducible, and then  $\Delta_N\subset \BS_\de(p_L)$ for all $N>N_1$, proving the claim.

Hence, by Proposition \ref{p41} and \eqref{e511} it follows that for any $\ep>0$ there exists  $\de>0$ such that
  \begin{align*}
  \dim_H(\OB\cap(p_L-\de, p_L+\de))&\ge(1-\ep)\dim_H\pi_{p_L}(\BS_\de(p_L))\\
  &\ge (1-\ep)\dim_H\pi_{p_L}(\Delta_N)\\
  &=(1-\ep)\frac{h_{top}(F_{\ws_{p_L,N}}(\al_1\ldots \al_m))}{\log p_L}\\
  &=(1-\ep)\frac{h_{top}(\ws_{p_L,N})}{\log p_L}
  \end{align*}
  for all sufficiently large integers $N>N_1$. 
  Letting $N\ra\f$ we obtain by Lemmas \ref{l54}  and \ref{l24} that
  \[
  \dim_H(\OB\cap(p_L-\de, p_L+\de))\ge (1-\ep)\frac{h_{top}(\vs_{p_L})}{\log p_L}=(1-\ep)\dim_H\u_{p_L}.
  \]
  Since $\ep>0$ was arbitrary, we complete the proof by letting $\ep\ra 0$.
\end{proof}

\begin{proof}
  [Proof of Theorem \ref{t12}]
 Take $q\in\OB$ and $\de>0$.  By Lemma \ref{l27}  there exists a sequence of plateaus $\set{[p_L(n), p_R(n)]}$ such that $p_L(n)$ converges to $q$ as $n\ra\f$. By Lemmas \ref{l56} and \ref{l57} it follows that
   \[
   \dim_H(\OB\cap(q-\de, q+\de))\ge \dim_H\u_{p_L(n)}
   \]
   for all sufficiently large $n$. Letting $n\ra\f$ and by Lemma \ref{l24}  we obtain that
   \begin{equation}\label{e513}
     \dim_H(\OB\cap(q-\de, q+\de))\ge \dim_H\u_{q}.
   \end{equation}
   Therefore, the theorem follows from \eqref{e513} and Proposition \ref{p52}.
\end{proof}

\section{Dimensional spectrum  of $\ub$}\label{sec:proof of th1}
Recall that $\ub$ is the set of univoque bases $q\in(1,M+1]$ for which $1$ has a unique $q$-expansion. In this section we will  use Theorem \ref{t12} to prove Theorem \ref{t14} for the dimensional spectrum of $\ub$, which states that
\[  \dim_H(\ub\cap(1, t])=\max_{q\le t} \dim_H\u_q \quad \text{ for all }~ t > 1. \]
We focus on $t \in (q_{KL}, M+1)$, since by Lemma \ref{l24}  the other cases are trivial.

 Since the proof of Lemma \ref{l43} above only uses properties of $\overline{\ub}$ instead of ${\OB}$, the proof also gives the following lemma for the set $\overline{\ub}$.
 \begin{lemma}
 \label{l61}
 Let $q\in\overline{\ub}\setminus\set{M+1}$. Then for any $\ep>0$ there exists a $\de>0$ such that
 \[
 \dim_H(\overline{\ub}\cap(q-\de, q+\de))\le (1+\ep)\dim_H\pi_{q+\de}(\us_\de(q)),
 \]
 where $\us_{\de}(q)=\set{\al(p): p\in\overline{\ub}\cap(q-\de, q+\de)}$.
 \end{lemma}

To prove Theorem \ref{t14} we first consider the upper bound.
\begin{lemma}\label{l62}
For any $t\in(q_{KL}, M+1)$ we have
\[
\dim_H(\overline{\ub}\cap(1, t])\le \max_{q\le t}\dim_H\u_q.
\]
\end{lemma}
\begin{proof}
Fix $\ep>0$, and take $t\in(q_{KL}, M+1)$. Then it suffices to prove
\begin{equation}\label{e61}
\dim_H(\overline{\ub}\cap(1, t])\le (1+\ep)(\max_{q\le t}\dim_H\u_q+\ep).
\end{equation}

By Lemmas \ref{l24} and \ref{l61} it follows that for each $q\in\overline{\ub}\cap(1, t]$ there exists a sufficiently small $\de=\de(q, \ep)>0$ such that
\begin{equation}\label{e62}
\begin{split}
 \dim_H\u_{q+\de}&\le  \dim_H\u_q+\ep,\\
\dim_H(\overline{\ub}\cap(q-\de, q+\de))&\le (1+\ep)\dim_H\pi_{q+\de}(\us_\de(q)).
\end{split}
\end{equation}
Observe that $\set{(q-\de, q+\de): q\in\overline{\ub}\cap(1, t]}$ is an open cover of $\overline{\ub}\cap(1, t]$, and that $\overline{\ub}\cap(1, t]=\overline{\ub}\cap[q_{KL}, t]$ is a compact set. Hence, there exist  $q_1, q_2,\ldots, q_N$ in $\overline{\ub}\cap(1, t]$ such that
\begin{equation}\label{e63}
\overline{\ub}\cap(1, t]\subseteq\bigcup_{i=1}^N\big(\overline{\ub}\cap(q_i-\de_i, q_i+\de_i)\big),
\end{equation}
where $\de_i=\de(q_i, \ep)$ for $1\le i\le N$.

Note by Lemmas \ref{l22} and \ref{l23}  that for each $i\in\set{1, 2,\ldots, N}$ we have
\[\pi_{q_i+\de_i}(\us_{\de_i}(q_i))=\pi_{q_i+\de_i}(\set{\al(p): p\in\overline{\ub}\cap(q_i-\de_i, q_i+\de_i)})\subseteq\u_{q_i+\de_i}.\]
 Then by (\ref{e62}) and (\ref{e63}) it follows that
\begin{align*}
\dim_H(\overline{\ub}\cap(1, t])&\le \dim_H\left(\bigcup_{i=1}^N\big(\overline{\ub}\cap(q_i-\de_i, q_i+\de_i)\big)\right)\\
&=\max_{1\le i\le N}\dim_H(\overline{\ub}\cap(q_i-\de_i, q_i+\de_i))\\
&\le (1+\ep)\max_{1\le i\le N}\dim_H\pi_{q_i+\de_i}(\us_{\de_i}(q_i))\\
&\le (1+\ep)\max_{1\le i\le N}\dim_H\u_{q_i+\de_i}\\
&\le (1+\ep)\max_{1\le i\le N}(\dim_H\u_{q_i}+\ep)\\
&\le(1+\ep)(\max_{q\le t}\dim_H\u_q+\ep).
\end{align*}
This proves (\ref{e61}), and completes the proof.
\end{proof}

The next lemma gives the lower bound of Theorem \ref{t14}.
\begin{lemma}\label{l63}
For any $t\in(q_{KL}, M+1)$ we have
\[
\dim_H(\overline{\ub}\cap(1, t])\ge \max_{q\le t}\dim_H\u_q.
\]
\end{lemma}
\begin{proof}
Take $t\in(q_{KL}, M+1)$.
Note by Lemma \ref{l24}   that the dimension function  $D: q\mapsto \dim_H\u_q$  is continuous. Then there exists  $q_*\in[q_{KL}, t]$ such that
\begin{equation*}
\dim_H\u_{q_*}=\max_{q\le t}\dim_H\u_q.
\end{equation*}
Since the entropy function $H$ is locally constant on the complement of $\B$, it follows by Lemma \ref{l24} that
\[
q_*\in(q_{KL}, t]\setminus\bigcup(p_L, p_R]\subseteq(q_{KL}, t]\cap\OB.
\]
If $q_*\in(q_{KL}, t)\cap\OB$, then the lemma follows by $\OB\subset\overline{\ub}$ and  Theorem \ref{t12}. If $q_*=t$, then  by Lemma \ref{l27} (i) there exists a sequence of plateaus $\set{[p_L(n), p_R(n)]}$ such that $p_L(n)\in(q_{KL}, t)\cap\OB$ and $p_L(n)\nearrow q_*$ as $n\ra\f$. Therefore, by Lemma \ref{l24} and Theorem \ref{t12} we also have 
\[
\dim_H(\overline{\ub}\cap(1, t])\ge\dim_H(\OB\cap(q_{KL}, t])\ge \dim_H\u_{p_L(n)}\;\ra\dim_H\u_{q_*}
\]
as $n\ra\f$. 
This establishes the lemma.
\end{proof}

\begin{proof}[Proof of Theorem \ref{t14}]
For $1 < t \le q_{KL}$ we have $\ub \cap (1,t] \subseteq\set{q_{KL}}$ and thus by Lemma \ref{l24} (i) it follows that
\[ \dim_H ( \ub \cap (1,t]) = 0 = \max_{q \le t} \dim_H \u_q.\]
For $t \ge M+1$ we have $\ub=\ub \cap (1,t]$ and the result also follows from Lemma \ref{l24}. For the remaining $t$ the result follows from Lemmas \ref{l62} and \ref{l63}, since $\overline{\ub}\backslash \ub$ is countable.

From Lemma \ref{l24} it follows that the dimension function $D: q \mapsto \dim_H\u_q$ has a Devil's staircase behavior (see also Remark \ref{r25} (1)). This implies that $\phi(t):=\max_{q\le t}\dim_H\u_q$ is a Devil's staircase in $(1, \f)$: {(i) $\phi$ is non-decreasing and continuous in $(1, \f)$; (ii) $\phi$ is locally constant almost everywhere in $(1, \f)$; and (iii) $\phi(q_{KL})=0$, and   $\phi(t)>0$ for any $t>q_{KL}$. }
\end{proof}

\section{Variations of $\ub(M)$}\label{sec:proof of th2}

For any $K \in \{0,1, \ldots, M\}$, let $\ub(K)$ denote the set of bases $q>1$ such that $1$ has a unique $q$-expansion over the alphabet $\set{0,1,\ldots,K}$. Then $\ub(K)\subset(1, K+1]$.  In this section we investigate the Hausdorff dimension of the intersection $\bigcap_{J=K}^M\ub(J)$, and prove Theorem \ref{t16}.  
Note that $q_{KL}=q_{KL}(M)$ is the smallest element of $\ub(M)$, and $K+1$ is the largest element  of $\ub(K)$. So, if $K+1<q_{KL}$ then $\ub(M)\cap\ub(K)=\emptyset$. Therefore, in the following we  assume $K\in[q_{KL}-1, M]$.
\begin{lemma}
  \label{l72}
Let $K\in[q_{KL}-1, M]$ be an integer. Then for each $q\in\ub(M)\cap(1,K+1]$ the unique expansion $\al(q)=(\al_i(q))$ satisfies
  \[
  M-K\le \al_i(q)\le K\qquad\textrm{for any}\quad i\ge 1.
  \]
\end{lemma}
\begin{proof}
Clearly, the lemma holds if $K=M$. So we assume $K<M$. Take $q\in\ub(M)\cap(1,K+1]\subseteq[q_{KL}, K+1]$. Then
    \[
    \al(q_{KL})\preceq \al(q)\preceq \al(K+1)=K^\f.
    \]
This, together with $\al_1(q_{KL})\ge M-\al_1(q_{KL})$, implies that
\[
     M-K\le \al_1(q_{KL})\le\al_1(q)\le K.
\]
Since $M>K$ and $q\in\ub(M)$, it follows from Lemma \ref{l23} (i) that
   \[M-K\le M-\al_1(q)\le \al_i(q)\le\al_1(q)\le K\qquad\textrm{for any}\quad i\ge 1.\]
   This completes the proof.
\end{proof}

\begin{lemma}
  \label{l73}
  Let $K\in[q_{KL}-1, M]$ be an integer. Then
  \[\ub(M)\cap\ub(K)=(1,K+1]\cap\ub(M).\]
\end{lemma}
\begin{proof}
Since $\ub(K)\subseteq(1,K+1]$, it suffices to prove that $\ub(M)\cap(1,K+1]\subseteq\ub(K)$. Take $q\in\ub(M)\cap(1,K+1]$. Then by Lemma \ref{l23} it follows that $\al(q)=(\al_i(q))$ satisfies
\begin{equation}\label{e71}
(K-\al_i(q)) \preceq (M-\al_i(q))\prec \al_{i+1}(q)\al_{i+2}(q)\cdots \prec \al(q)\quad\textrm{for all}\quad i\ge 1.
  \end{equation}
Note by Lemma \ref{l72} that $0\le \al_i(q)\le K$ for all $i\ge 1$. Hence, by   (\ref{e71}) and Lemma \ref{l23} we conclude that $q\in\ub(K)$.
\end{proof}

\begin{proof} [Proof of Theorem \ref{t16}]
First we prove (i). Clearly, if $K<q_{KL}-1$ then $\bigcap_{J=K}^M\ub(J)=\emptyset$, and therefore (i) holds by Lemma \ref{l24} (i). If $q_{KL}-1\le K\le M$, then by repeatedly using Lemma \ref{l73} we conclude that
\begin{align*}
\bigcap_{J=K}^M\ub(J)&=\big(\ub(M)\cap\ub(M-1)\big)\cap\bigcap_{J=K}^{M-2}\ub(J)\\
&=(1,M]\cap\ub(M)\cap\bigcap_{J=K}^{M-2}\ub(J)\\
&=(1,M]\cap\big(\ub(M)\cap\ub(M-2)\big)\cap\bigcap_{J=K}^{M-3}\ub(J)\\
&= (1,M-1]\cap\ub(M)\cap\bigcap_{J=K}^{M-3}\ub(J)\\
&\quad\cdots\\
&=(1,K+1]\cap\ub(M).
\end{align*}
Therefore, by Theorem \ref{t14} we establish  (i).

As for (ii), we observe that for any $L\ge 1$,
\begin{equation}\label{e72}
\ub(L)=\Big(\ub(L)\setminus\bigcup_{J\ne L}\ub(J)\Big)\cup \bigcup_{J\ne L}\big(\ub(L)\cap\ub(J)\big).
\end{equation}
From (i) and Lemma \ref{l24} (i) it follows that $\dim_H(\ub(L)\cap\ub(J))<1$ for any $J\ne L$. Furthermore, by Lemma  \ref{l26} we have $\dim_H\ub(L)=1$ (see also, \cite[Theorem 1.6]{Komornik_Kong_Li_2015_1}). Therefore, (ii) immediately follows from \eqref{e72}.
\end{proof}

\section{Final remarks}
It was shown in Theorem \ref{t14} that the function $\phi(t)=\dim_H (\ub\cap(1, t])$ is a Devil's staircase in $(1,\f)$ (see Figure \ref{fig:1} for the sketch plot of $\phi$). Then  a natural question is to ask about the presence and position of plateaus for $\phi$, i.e., maximal intervals on which $\phi$ is constant. By Lemma \ref{l24} (i) and Theorem \ref{t14} it follows that $\phi(t)=0$ if and only if $t\le q_{KL}$, and $\phi(t)=1$ if and only if $t\ge M+1$. Hence, the first plateau of $\phi$ is $(1, q_{KL}]$, and the last plateau is $[M+1, \f)$.

Since $\phi(t)=\max_{q\le t}\dim_H\u_q$, an interval $[q_L, q_R]$ is a plateau of $\phi$ if and only if
 \begin{align*}
    \dim_H\u_p< \dim_H\u_{q_L} &&\textrm{for any}&& p<q_L,\\
   \dim_H\u_q\le \dim_H\u_{q_L} && \textrm{for any}&&  q_L\le q\le q_R,\\
   \dim_H\u_r> \dim_H\u_{q_L}& &\textrm{for any}&& r>q_R.
 \end{align*}
 By Lemma \ref{l24} it follows that for each plateau $[q_L, q_R]$ of $\phi$ we have $\dim_H\u_{q_L}=\dim_H\u_{q_R}$.

  {\bf Question 1}. Can we describe the plateaus of $\phi$ in $(q_{KL}, M+1)$?

Theorem \ref{t14} tells us that the set $\ub$ gets heavier towards the right, but does not say anything about the local weight.

 {\bf Question 2}. For any $t_2>t_1>1$, what is the local dimension $\dim_H(\ub\cap[t_1, t_2])$?

\subsection*{Acknowledgements}
{The authors thank the anonymous referee for many useful suggestions. The second author was supported by   NSFC No.~11401516. The third author was
supported by  NSFC No.~11671147, 11571144  and Science and Technology Commission of Shanghai Municipality (STCSM)  No.~18dz2271000. The forth  author was supported by NSFC No.~11601358.}

%

\end{document}